\documentclass[reqno,11pt]{amsart}
\usepackage{amsmath, latexsym, amsfonts, stmaryrd, amssymb, amsthm, amscd, array, caption}
\usepackage[utf8]{inputenc}
\usepackage{graphics,epsf,psfrag,epsfig}
\usepackage[colorlinks=true, pdfstartview=FitV, linkcolor=blue, citecolor=blue, urlcolor=blue,pagebackref=false]{hyperref}

\usepackage[showlabels,sections,floats,textmath,displaymath]{}

\numberwithin{equation}{section}
\newtheorem{theorem}{Theorem}[section]
\newtheorem{lemma}[theorem]{Lemma}
\newtheorem{proposition}[theorem]{Proposition}

\newtheorem{rem}[theorem]{Remark}
\newtheorem{claim}[theorem]{Claim}
\newtheorem{definition}[theorem]{Definition}
\newtheorem{hyp}[theorem]{Assumption}

\newtheorem{conjecture}[theorem]{Conjecture}

\renewenvironment{proof}[1][Proof]{\begin{trivlist}
\item[\hskip \labelsep {\bfseries #1}]}{\qed\end{trivlist}}

\newcommand{\ind}{\mathbf{1}}

\renewcommand{\le}{\leq}

\newcommand{\Z}{\mathbb{Z}}
\newcommand{\N}{\mathbb{N}}
\renewcommand{\tilde}{\widetilde}
\renewcommand{\hat}{\widehat}

\DeclareMathSymbol{\leqslant}{\mathalpha}{AMSa}{"36} 
\DeclareMathSymbol{\geqslant}{\mathalpha}{AMSa}{"3E} 
\DeclareMathSymbol{\eset}{\mathalpha}{AMSb}{"3F}     
\renewcommand{\leq}{\;\leqslant\;}                   
\renewcommand{\geq}{\;\geqslant\;}                   
\newcommand{\dd}{\,\text{\rm d}}             

\newcommand{\cE}{{\ensuremath{\mathcal E}} }

\newcommand{\cN}{{\ensuremath{\mathcal N}} }

\newcommand{\cT}{{\ensuremath{\mathcal T}} }

\newcommand{\bP}{{\ensuremath{\mathbf P}} }
\newcommand{\bE}{{\ensuremath{\mathbf E}} }

\newcommand{\bbE}{{\ensuremath{\mathbb E}} }

\newcommand{\bbN}{{\ensuremath{\mathbb N}} }

\newcommand{\bbP}{{\ensuremath{\mathbb P}} }

\newcommand{\bbR}{{\ensuremath{\mathbb R}} }

\newcommand{\bbZ}{{\ensuremath{\mathbb Z}} }

\newcommand{\ga}{\alpha}
\newcommand{\gb}{\beta}
\newcommand{\gga}{\gamma}            
\newcommand{\gd}{\delta}
\newcommand{\gep}{\varepsilon}       

\newcommand{\go}{\omega}

\newcommand{\gl}{\lambda}

\newcommand{\gU}{\Upsilon}

\makeatletter
\def\captionfont@{\footnotesize}
\def\captionheadfont@{\scshape}

\long\def\@makecaption#1#2{%
  \vspace{2mm}
  \setbox\@tempboxa\vbox{\color@setgroup
    \advance\hsize-6pc\noindent
    \captionfont@\captionheadfont@#1\@xp\@ifnotempty\@xp
        {\@cdr#2\@nil}{.\captionfont@\upshape\enspace#2}%
    \unskip\kern-6pc\par
    \global\setbox\@ne\lastbox\color@endgroup}%
  \ifhbox\@ne 
    \setbox\@ne\hbox{\unhbox\@ne\unskip\unskip\unpenalty\unkern}%
  \fi
  \ifdim\wd\@tempboxa=\z@ 
    \setbox\@ne\hbox to\columnwidth{\hss\kern-6pc\box\@ne\hss}%
  \else 
    \setbox\@ne\vbox{\unvbox\@tempboxa\parskip\z@skip
        \noindent\unhbox\@ne\advance\hsize-6pc\par}%
\fi
  \ifnum\@tempcnta<64 
    \addvspace\abovecaptionskip
    \moveright 3pc\box\@ne
  \else 
    \moveright 3pc\box\@ne
    \nobreak
    \vskip\belowcaptionskip
  \fi
\relax
}
\makeatother
\def\writefig#1 #2 #3 {\rlap{\kern #1 truecm
\raise #2 truecm \hbox{#3}}}

\newcommand{\pin}{\mathrm{pin}}
\newcommand{\q}{\mathrm{que}}
\newcommand{\tf}{\mathtt{F}}
\renewcommand{\P}{{\ensuremath{\mathbf P}} }
\newcommand{\E}{{\ensuremath{\mathbf E}} }
\newcommand{\Eo}{\mathbb{E}}
\newcommand{\Po}{\mathbb{P}}

\renewcommand{\a}{ \mathrm{a}}

\newcommand{\K}{\mathrm{K}}

\newcommand{\cS}{\mathcal{S}}

\newcommand{\Fav}{ \mathfrak{F}}

\newcommand{\W}{\mathsf{W}}

\author[Q. Berger]{Quentin  Berger
\\
\textit{\tiny University of Southern California}}

\title[Appearance of an infinite disorder regime for pinning model]{Pinning model in random correlated environment:
appearance of an \textsl{infinite disorder} regime}

\begin{document}

\begin{abstract}

  We study  the influence of a correlated disorder
  on the localization phase transition in the pinning model \cite{GBbook}. When correlations are strong enough, an \textsl{infinite disorder} regime arises: large and frequent attractive regions appear in the environment.  We present here a pinning model in random binary ($\{-1,1\}$-valued) environment.  Defining \textsl{infinite disorder} via the requirement that the probability of the occurrence of a large attractive region is \emph{sub-exponential} in its size, we prove that it coincides with the fact that the critical point is equal to its minimal possible value, namely $h_c(\gb)=-\gb$.  We also stress that in the  infinite disorder regime, the phase transition is smoother than in the  homogeneous case, \emph{whatever the critical exponent of the homogeneous model is}: disorder is therefore \textsl{always relevant}.  We
 illustrate these results with the example of an environment based on the sign of a Gaussian correlated sequence, in which we show that the phase transition is of infinite order in presence of infinite disorder. Our results contrast with results known in the literature, in particular in the case of an IID disorder, where the question of the influence of disorder on the critical properties is answered via the so-called Harris criterion, and where a conventional relevance/irrelevance picture holds.
 \\
 \\
 2010 \textit{Mathematics Subject Classification: 82D60, 60K37, 60K05}
 \\
 Keywords:\textit{ Polymer pinning, Disordered systems, Critical
   phenomena, Correlation, Wein\-rib-Halperin prediction, Infinite
   disorder.}
\end{abstract}

\maketitle

\section{Introduction}

\subsection{Physical motivations}
\label{sec:introcorrel}

In the study
of critical phenomena,
a fundamental question is that of the influence of
a quenched randomness on the physical properties of a system
when approaching criticality.
More precisely, if a disordered system is shown to undergo a phase transition,
one compares its behavior close to the critical point to that of the
non-disordered (or homogeneous) model.
If the features of the phase transition (essentially the critical exponents) are
changed by the presence of randomness,
disorder is said to be \textsl{relevant}.
The Harris criterion (see \cite{Harris}) gives a prediction
for disorder irrelevance for
$d$-dimensional disordered systems, when randomness has short-range correlations:
disorder is irrelevant if $d\nu^{\rm pur}>2$, where $\nu^{\rm pur}$ 
is the correlation length critical exponent of the homogeneous model. If this condition is not fulfilled, then the critical behavior must change: in particular it has been shown in \cite{CCFS86} that the  critical exponent of the correlation length (suitably defined, in terms of finite-size scaling) is larger than $2/d$ for the disordered system. When disorder presents  long-range correlations, one invokes the Weinrib-Halperin prediction \cite{WeinHalp83}.
With correlations between two couplings at $i$ and $j$ decaying like $|i-j|^{-a}$, $a>0$, one should have that the condition for disorder irrelevance becomes $\min(d,a) \nu^{\rm pur}>2$: the Harris prediction is changed only if $a<d$.

When dealing with this question of the influence of disorder on critical properties, the effect of rare regions with atypical disorder reveals to be crucial, cf. \cite{Vojta06}. To simplify the statements,
if the rare attractive regions are too spread out, their effect on the system remains bounded when considering larger and larger length scale: in the renormalization group language, the fixed point is at a finite randomness. One then lies in what we call the \textsl{conventional} regime, where Harris' and Weinrib-Halperin's ideas should be applicable to decide disorder relevance/irrelevance. On the other hand, if atypical regions occur very frequently, their impact on the system increases without limit when taking larger and larger length scale: the system is said to be governed by an infinite-randomness fixed point, and one lies in the \textsl{infinite disorder} regime. Infinite disorder features have been studied for various systems, essentially in random quantum Ising model (starting with Fisher \cite{Fisher92,Fisher94}),
and we refer to \cite{IgloMont05} for a review.

We stress that, in \cite[\S 3.3]{Vojta06}, a classification of the effects of rare regions is proposed, depending on how the contribution of a single atypical region to observables increases with its size. In class A, the effect of rare regions is overcome by their exponentially small density, and their effect on the phase transition is marginal. In class B, the contribution of atypical regions increases exponentially with their size, which plays an important role in the global phase transition.
Finally, in class C, the phase transition is actually destructed by smearing: the contribution of regions with atypical disorder overwhelms their dispersal, and the behavior of the global system is corrupted.

\smallskip

In the mathematical literature,
one class of models has been given much attention lately regarding
the question of disorder relevance/ir\-rel\-ev\-ance:
the disordered pinning model (see \cite{GBbook, SFLN, denHoll}),
that models the adsorption of a polymer on a wall or a defect line.
Its advantage is that one can play on $\nu^{\rm pur}$ as a parameter, to cover the whole range of the relevance/irrelevance picture (in the conventional regime). 
In the IID case, the Harris criterion has been proven, thanks to a series of papers \cite{A06, AZ08, CdH10, DGLT07, GLT09, GT05, L10, T06}. More recently, the case of long-range correlated disorder has also been attacked, mostly in the case of Gaussian correlated disorder \cite{Bcorrel,BThier,Poisat12}, with correlation decay exponent $a>0$;
and some steps were made towards the Weinrib-Halperin criterion in the case $a>1$ (\textit{i.e.}\
when the Harris criterion should be unchanged).
Infinite disorder arises (although it was not mentioned in these terms) in the case $a<1$: the phase transition has been shown to disappear  (the critical point equals $-\infty$), precisely because of large disorder fluctuations, as in Class C of \cite{Vojta06}.


In this article, we focus on the appearance of a \emph{infinite disorder} regime for pinning models. With a choice of non-Gaussian (actually bounded) environment and strong enough correlations, we show that the phase transition survives but the relevance/irrelevance picture is drastically modified, and we present some of the unusual characteristics of the critical behavior of the system, similar to the smearing of the phase transition described in Class C of \cite{Vojta06}.

\subsection{Definition of the pinning model}
\label{sec:pinning}

Let $\tau:=(\tau_i)_{i\geq 0}$ be a renewal process, with law $\P$: $\tau$
is a sequence with
$\tau_0=0$ $\bP$-a.s., and such that the variables $(\tau_i-\tau_{i-1})_{i\in\bbN}$
are IID with support in $\bbN$,
with common law that is called inter-arrival distribution.

\begin{hyp}
\label{assumpK}
We make the assumption that the renewal process is recurrent, that is $\bP(\tau_1=+\infty)=0$.
We assume that the
inter-arrival distribution, denoted by $\K(\cdot)$,
satisfies
\begin{equation}
\label{K}
\K(n):=\bP(\tau_1=n)\stackrel{n\to\infty}{=} (1+o(1))\frac{c_K}{n^{1+\ga}},
\end{equation} 
for some $\alpha>0$ and $c_K>0$.
We also note for convenience $\bar\K(n) := \bP(\tau_1> n)$.
\end{hyp}

There is a physical interpretation for
the set $\tau:=\{\tau_0,\tau_1,\ldots\}$ (we make here a slight abuse of notations):
it can be thought
as the set of contact points between a polymer and a defect line.
Indeed, if $S$ is a random walk on $\bbZ^{\texttt{d}}$,
the graph of the directed random walk $(k,S_k)_{k\in[0,N]}$
represents a $1$-dimensional polymer chain living in a $(1+\texttt{d})$-dimensional space,
and interacting with the defect line $\bbN\times\{0\}$.  Then, the 
set of return times to $0$ of the random walk $S$ is a renewal process.
In the case of a $\texttt{d}$-dimensional simple random walk, the
inter-arrival distribution is known to satisfy the asymptotic \eqref{K} (up to the aperiodicity condition, which is easily overcome):
one has $\ga=1/2$ for $\texttt{d}=1$, $\ga=0$ (with logarithmic correction)
if $\texttt{d}=2$, and $\ga=\texttt{d}/2-1$ if $\texttt{d}\geq3$, see \cite{DonKur}.
Note that one can actually consider transient renewals, see Remark \ref{rem:reccurence}: the recurrence of the renewal does not lead to any loss of generality.

\subsubsection{The disordered model}

Let us consider a random sequence $\go=\{\go_i\}_{i\in\N}$ 
whose law, denoted $\bbP$, is ergodic and such that $\bbE[|\go_1|]<+\infty$.
Given the sequence $\go$ (the environment) and parameters $h\in \bbR,\gb\geq 0$, we define
the \textsl{quenched}
polymer measure with {\sl free boundary condition}.
It is a Gibbs transformation of the law~$\P$, up to length $N$:
\begin{equation}
\label{freepolymmeas}
 \frac{\dd \bP_{N,h}^{\go,\gb}}{\dd \bP} (\tau) := 
\frac{1}{Z_{N,h}^{\go,\gb}}   \exp\left( \sum_{n=1}^N  (h+\gb\go_n) \gd_n\right),
\end{equation}
with the notation $\gd_n:=\ind_{\{n\in\tau\}}$. The quantity
$ Z_{N,h}^{\go,\gb}:= \bE\left[ \exp\left( \sum_{n=1}^N  (h+\gb\go_n) \gd_n\right) \right]$
is used to normalize $\bP_{N,h}^{\go,\gb}$
to a probability measure, and is called
the \emph{partition function} of the disordered system.

\begin{rem}\rm
\label{rem:reccurence}
Note that if the underlying renewal were transient, one would consider the recurrent renewal, with inter-arrival distribution $\tilde K (n)=\tilde\bP(\tilde\tau_1=n) := K(n) / \bP(\tau_1<+\infty)$. Then, one has that
$ Z_{N,h}^{\go,\gb}:= \tilde\bE\left[ \exp\left( \sum_{n=1}^N  (h+\log(\bP(\tau_1<+\infty))+\gb\go_n) \gd_n\right) \right]$: one gets back to studying a recurrent renewal only thanks to a change of parameters $h\mapsto h+\log(\bP(\tau_1<+\infty))$.
\end{rem}

The polymer measure defined in \eqref{freepolymmeas} then
corresponds to giving an energy reward (or penalty, depending on its sign) to the trajectory of the renewal when it touches the defect line, at the times $\{\tau_i\}_{i\in\bbN}$. The interaction is composed of a homogeneous reward, $h$, and an inhomogeneous one, $\gb\go_n$.

\begin{proposition}
\label{prop:existF}
One defines the \emph{quenched} free energy of the system
\begin{equation}
\label{defF}
\tf(\gb,h):=\lim_{N\to\infty} \frac{1}{N}\log  Z_{N,h}^{\go,\gb} = \lim_{N\to\infty} \frac{1}{N}\bbE \log  Z_{N,h}^{\go,\gb},
\end{equation}
which exists
and is $\bbP$-a.s.\ constant.
The map $h\mapsto \tf(\gb,h)$ is convex, non-negative and non-decreasing.
Thus there exists a ({\sl quenched}) critical point $h_c(\gb)$,
possibly infinite, for which one has that $\tf(\gb,h)>0$
if and only if $h>h_c(\gb)$.
\end{proposition}
This is a classical result for pinning models (see \cite[Ch. 4]{GBbook}), and
we do not prove it here.
The free energy, or energy per monomer,
carries physical information on the system: it is easy to see that $\frac{\partial \tf}{\partial h}$ (when it exists)
is the the asymptotic density of contacts under $\bP_{N,h}^{\go,\gb}$.
If $\tf(\gb,h)>0$ then
there is a positive density of contacts:  trajectories stick to the defect line.
If $\tf(\gb,h)=0$ there is
a null density of contacts: trajectories
are wandering away from the defect line.
Therefore, a phase transition occurs at the \emph{quenched}
critical point $h_c(\gb)$,
from a \textsl{delocalized phase} for $h<h_c(\gb)$
to a \textsl{localized phase} for $h>h_c(\gb)$.

\smallskip

We also define the {\sl annealed} system, that is often compared  with the disordered system:
the {\sl annealed} partition function is $ Z_{N,h,\gb}^{\a}:=\bbE[Z_{N,h}^{\go,\gb}]$,
and the {\sl annealed} free energy is
  $\tf^{\a}(\gb,h):=\lim_{N\to\infty} \frac{1}{N}\log  \bbE Z_{N,h}^{\go,\gb}$.
One also has an {\sl annealed} critical point $h_c^{\a}(\gb)$ that separates
phases where $\tf^{\a}(\gb,h)=0$ and where $\tf^{\a}(\gb,h)>0$.
It is straightforward to obtain from the Jensen inequality that $\tf^{\a}(\gb,h)\geq \tf(\gb,h)$,
which gives $h_c^{\a}(\gb)\leq h_c(\gb)$.
When the inequality is strict, we have an indication that disorder is relevant.

The question of disorder relevance/irrelevance is therefore asked both
in terms of critical exponents, comparing the disordered and homogeneous critical behavior,
and in terms of
critical points, comparing the quenched and annealed ones.

\subsubsection{Reminder on the homogeneous pinning model}

Let us consider the homogeneous (or pure) model, and its partition function, that we denote
$Z_{N,h}^{\rm pur}:= \bE\left[  \exp\left(h\sum_{n=1}^N \gd_n\right) \right]$.
The particularity of the homogeneous pinning model is that it is exactly solvable, see \cite{Fisher}.

\begin{proposition}[Critical behavior of the homogeneous model]
\label{prophomo}
Under Assumption \ref{assumpK}, one has that
$h_c:=h_c(0)=0$.
The behavior of the free energy (denoted $\tf(h)$ for simplicity of the statements) for $h$ close to $h_c=0$
is
\begin{equation}
\tf(0,h)=:\tf(h) \stackrel{h\searrow0}{\sim}
  \begin{cases}
     \left( \frac{\ga}{\Gamma(1-\ga) c_{\K}}\right)^{1/\ga}  \, h^{1/\ga}  &\quad \text{if } \ga<1, \\
      \frac{1}{c_{\K}}\, |\log h|^{-1} h  & \quad \text{if } \ga=1, \\
      \left(  \sum_{n\in\N} nK(n)\right)^{-1} \, h & \quad \text{if } \ga>1.
  \end{cases}
\end{equation}
\end{proposition}
This proposition tells that the critical exponent of the pure free energy is
$\nu^{\rm pur}=1\vee 1/\ga$ (we note $a\vee b = \max(a,b)$ and $a\wedge b=\min(a,b)$),
leaving aside the $\log$ factor in the case $\ga=1$. 
Let us mention that
$\nu^{\rm pur}$ is also the critical exponent of the correlation length, cf.\ \cite{G_correl}.
In the sequel, we actually do not treat
the case $\ga=1$ only to avoid too many technicalities (this case is not fundamentally different).

\subsection{Review of the known results}

\subsubsection{Case of IID environment}

The picture of disorder relevance/irrelevance is now mathematically understood,
the marginal case $\nu^{\rm pur}=2$ ($\ga=1/2$) being also settled.
We collect the results (cf.\ \cite{SFLN} for an overview), predicted by the Harris criterion:

\textbullet\ 
If $\alpha<1/2$, disorder is \emph{irrelevant}: for $\gb$ small enough, one has
 $  h_c(\gb) = h_c^{\a}(\gb)$, and the order of the (disordered) phase transition is the same
 as for the pure system,
 see \cite{A06,L10,T08}.

\textbullet\  If $\alpha\geq 1/2$, disorder is \emph{relevant}: one has
that  $h_c(\gb)>h_c^{\a}(\gb)$  for all $\gb>0$, see\cite{AZ08,DGLT07,GLT09} (bounds on the gap between the critical points are also given). Moreover, the order of the (disordered) phase transition
is shown to be at least $2$ \cite{GT05}, showing disorder relevance when $\nu^{\rm pur}<2$ ($\ga>1/2$).

We also mention that a new approach to this problem has been developed
recently. It relies on a Large Deviation Principle for a process of cutting words into a letter sequence \cite{BGdH08},
and has been fruitful in many contexts, and in particular for pinning models \cite{CdH10}.

\subsubsection{Case of a correlated environment}

The first natural type of correlated environment to be considered is the Gaussian one:
$\go=\{\go_n\}_{n\in\bbN}$ is a centered Gaussian stationary sequence, with
covariance function $\bbE[\go_i \go_{i+n}]=:\rho_n$.
Thanks to the Gaussian structure of the correlations,
one is able to compute the annealed partition function,
\begin{equation}
\bbE\left[ Z_{N,h}^{\go,\gb}\right]
=\bE\left[ \exp\left( h\sum_{i=1}^{N}\gd_i 
      + \frac{\gb^2}{2} \sum_{1\leq i, j \leq N} \rho_{|j-i|} \gd_i \gd_j\right) \right].
\end{equation}
Note that in the case of an IID environment,
the annealed system
is just the homogeneous pinning model with parameter $h+\frac{\gb^2}{2} \bbE[\go_1^2]$.
With correlations, the annealed model
is much more difficult
to solve!

The case of  finite-range correlations is treated in
 \cite{Poisat1, Poisat2}, and reserves no surprises: the Harris criterion is still valid.
Let us now present the result for long-range, power-law decaying correlations: $\rho_n\stackrel{n\to\infty}{\sim} c n^{-a}$,
for some $a>0$ and some constant $c>0$.
In \cite{Bcorrel} and \cite{Poisat12},
the authors show that if 
$a>2$, then the annealed critical exponent is equal to $\nu^{\rm pur}$.
In \cite{Bcorrel}, it is also proven that when $a>1$, the phase transition is of order
at least $2$, proving disorder relevance when $\ga>1/2$ ($\nu^{\rm pur}<2$),
as predicted.
In \cite{BThier}, the authors treat the hierarchical version of this model,
and prove the Weinrib-Halperin prediction in the case $a>1$,
both in terms of critical points,
and in terms of
critical exponents.
The case $a<1$ appears to be more problematic. 
Both in the standard and the hierarchical
version of the model (\cite{Bcorrel} and \cite{BThier}),
the annealed free energy is infinite, and
$\tf(\gb,h)>0$ for all $\gb>0$ and $h\in\bbR$.
There is no phase transition anymore
($h_c(\gb)=-\infty$), and 
it is therefore not possible to study the influence of disorder on the phase transition.
This phenomenon is due to the fact that when $a<1$,
there are large and frequent regions in the environment that are arbitrarily attractive.
($\go_n$ can be arbitrarily large).
We then speak of "infinite disorder" (a more precise statement is made in Definition \ref{defstrong}).

An idea to by-pass this problem is to consider a bounded environment, so that the phase transition
occurs, at some finite critical point. The inconvenient is that one has to
abandon the Gaussian character of the environment.
In \cite{BLpin},
the authors construct an ad-hoc binary environment:
it has blocks of $0$ and of $-1$'s, where the sizes of the different blocks are independent,
with power-law tail distribution, of exponent $\vartheta>1$ (the case $\vartheta<1$ being trivial).
It is proven that for all $\gb>0$, the critical point is equal to its minimal possible value
$h_c(\gb)=0$. The sharp critical behavior of the free energy is also given,
the critical exponent being $\nu^{\q}=\vartheta \nu^{\rm pur}>\nu^{\rm pur} $ (with explicit logarithmic corrections): disorder is relevant irrespective of the value of $\nu^{\rm pur}$. Moreover, one remarks from this example that the presence of "strong disorder" is not characterized in terms of the power-law decay exponent of the two point correlations function (the central quantity for the Weinrib-Halperin prediction): indeed, the correlation between $\omega_i$ and $\omega_{i+n}$ decays like $n^{-a}=n^{-(\vartheta-1)}$ and, even when  $\vartheta-1>1$, the Harris criterion for relevance/irrelevance fails (in contrast with the  Weinrib-Halperin prediction).

\subsection{Outline of the results}

In this paper, we consider a bounded (correlated) environment,
and we actually focus on the choice of a binary environment, $\go\in\{-1,+1\}^{\bbN}$, which is also assumed to be stationary and ergodic. We now define precisely what we mean by {\sl infinite disorder}. It is characterized by the fact that favorable regions are very large and frequent: the distance between two attractive regions (\textit{i.e.}\ constituted of only $+1$) of size larger than $n$ is subexponential in $n$. From the ergodicity of the environment, it is enough to consider the exponential decay of $\cT_1(n)$, the distance to the origin of the first attractive region of size larger than $n$.
\begin{definition}\rm
\label{defstrong}
If $\liminf_{n\to\infty} \frac{1}{n}\log \cT_1(n) =0$ $\bbP$-p.s.,
then we say that one has \emph{infinite disorder}.
Moreover, \emph{infinite disorder} is characterized by
\begin{equation}
\label{charactstrong}
\liminf_{n\to\infty} \frac{1}{n} \log \cT_1(n) =0\ \ \bbP-p.s.\ \Leftrightarrow \ \
\liminf_{n\to\infty} -\frac{1}{n} \log \bbP(\go_1=+1,\ldots,\go_n=+1) =0.
\end{equation}
\end{definition}
The presence of infinite disorder is then read
in terms of subexponential decay of the probability
to have a completely attractive environment of size $n$.
We prove \eqref{charactstrong} in Section \ref{sec:comments}, see Lemma \ref{lem:charactstrong}.
The question that we address here is to determine
how the critical properties of
the phase transition are modified in the infinite disorder regime, with respect to the pure case.
Since
\begin{equation}
Z_{N,h}^{\go,\gb}= \bE\left[ e^{ \sum_{i=1}^N (\gb\go_i+h) \gd_i } \right]\leq  \bE\left[ e^{ (h+\gb) \sum_{i=1}^N \gd_i } \right] =Z_{N,h+\gb}^{\rm pur},
\end{equation}
one has $\tf(\gb,h)\leq \tf(0,h+\gb)$
and $h_c(\gb)\geq -\gb$ (since $h_c(0)=0$).
Our first theorem then identifies the appearance of a infinite disorder regime by the fact that
the critical point is equal to its minimal possible value, $-\gb$.
\begin{theorem}
\label{thm:outline1}
Under the assumption that correlations are non-increasing (in the sense of Assumption \ref{hyp1}),
then the following criterion holds:
\begin{center}
$ h_c(\gb)=-\gb$ for every $\gb>0$ if and only if one lies in the infinite disorder regime.
\end{center}
\end{theorem}

A more precise statement is made in Theorem \ref{thm:criterion}.
We guess (see Conjecture \ref{conj}) that this result is  true even without Assumption \ref{hyp1} (remark that
 Theorem \ref{thm:criticpoint} below says that the implication ``infinite disorder implies $h_c(\gb)=-\gb$'' does not require such assumption).
As far as the features of the phase transition are concerned,
Theorem \ref{thm:stronglyrel} yields that in
presence of infinite disorder, 
the quenched free energy has always a smoother critical behavior
than that of the pure case: disorder is relevant
irrespective of the value of $\nu^{\rm pur}$,
and is said to be \emph{strongly relevant}.
\begin{theorem}
\label{thm:stronglyrel}
If $h_c(\gb)=-\gb$ (and in particular in presence of infinite disorder),
then for all $\gb>0$, one has
\begin{equation}
\tf(\gb,-\gb+u) \stackrel{u\searrow0}{=} o(\tf(0,u)).
\end{equation}
\end{theorem}

We also give bounds on the free energy in a very
general setting, see Proposition \ref{prop:boundsF}.
To complete the picture, we present a natural example,
the \emph{Gaussian signs} environment:  the $\{-1,1\}$-valued
sequence $\go$ is simply based on the sign of a correlated Gaussian sequence
(with power-law decaying correlations, of decay exponent $a$).
Infinite disorder regime appears when $a<1$,
and the phase transition is then of infinite order.

\smallskip
Let us now highlight the organization of the paper.
In Section \ref{sec:results}, we present some useful
notations, and expose our main results,
 that we emphasize thanks to the Gaussian signs example. We comment these results
point by point in Section \ref{sec:comments}.
In Section \ref{sec:annealed}, we discuss the annealed model:
it exhibits an unconventional behavior in the infinite disorder regime,
and enables us to prove parts of our results.
In Section \ref{sec:proof}, we prove the remaining results, giving
lower bounds and upper bounds on the free energy.
In Appendix, we prove some of the Gaussian estimates needed for the Gaussian signs example,
and Lemma \ref{lembasic} on the homogeneous model.

\section{Main results}
\label{sec:results}

\subsection{First notations and results}
\label{sec:notations}

We consider a sequence $\go=\{\go_i\}_{i\geq -1}$ (we choose $i\geq-1$
instead of $i\in\N$ for notation convenience, see the following definitions),
and we assume that $\go$ is ergodic and $\{-1,1\}$-valued (we note abusively $\go\in\{-1,1\}^{\N}$).
We also take $\go$ 
non-trivial, in the sense that $\bbP(\go_1=+1)>0$ and $\bbP(\go_1=-1)>0$.

Our environment is then composed of
favorable and unfavorable regions, whether $\go_i=+1$ or $\go_i=-1$.
For every set of indices $\cE=\{i_1,\ldots,i_n\}$,
we define the event
\begin{equation}
\Fav_{\cE}=\Fav_{\{i_1,\ldots,i_n\}}:= \left\{ \go_{i_1}=+1,\ldots,\go_{i_n}=+1  \right\}
\end{equation}
that the environment is attractive at sites $i_1,\ldots,i_n$.
With Definition \ref{defstrong},
one characterizes infinite disorder by the subexponential
decay of $\bbP(\Fav_{\llbracket 1,n \rrbracket})$.

Given the environment $\go=\{\go_i\}_{i\geq-1}$,
we condition it to have $\go_{-1}=-1$, $\go_0=+1$ (which has positive probability, so that the free energy is not affected).
We then define the sequences $(T_{n})_{n\geq 0}$
and $(\xi_n)_{n\geq 1}$ iteratively, setting $T_0:=0$, and for all $n\geq 1$
\begin{equation}
\label{defxi}
\begin{split}
  T_{n}&:=\inf \{i> T_{n-1}\ ;\ \go_{i+1}\neq \go_{i}\},\\
  \xi_n & := T_{n}-T_{n-1}
\end{split}
\end{equation}
Thus our system is cut into segments of size $\xi_n$, on which
$\go$ is constant valued, equal alternatively to $+1$ and to $-1$
(we write $\go\equiv+1$ and  $\go\equiv -1$).
The choice of conditioning to $\go_{-1}=-1,\go_0=+1$
enables us to identify the blocks with odd indices $(T_{2j},T_{2j+1}]$ for $j\geq0$
(therefore of size $\xi_{2j+1}$),
as the attractive ones.

\begin{rem}\rm
\label{rem:xifini}
The ergodicity and non triviality
of the sequence $\go$ implies that $\bbE[\xi_1]<+\infty$ and $\bbE[\xi_2]<+\infty$.
Indeed, one considers the (ergodic) sequence
$\bar\go=\{\bar\go_n\}_{n\geq0}$ defined by $\bar\go_n:=(\go_{n-1},\go_{n})$
for all $n\geq 0$.
This sequence is, according to our notations, conditioned to
start with $\bar\go_0=(-1,1)$.
Then in \cite[Ch.I.2.c]{Shields}, the \emph{generalized renewal process} of the set $\{(-1,1)\}$
is defined as the sequence of indices $\{k\geq 0, \bar\go_k=(-1,1)\}= \{T_{2j}\}_{j\geq0}$, 
and the \emph{return-time process} to the set $\{(-1,1)\}$ (of positive measure) is defined as the sequence 
$\{T_{2j}-T_{2(j-1)}\}_{j\in\N} = \{\xi_{2j-1}+\xi_{2j}\}_{j\in\N}$.
It is shown that the return-time process is ergodic (see \cite[Th.I.2.19]{Shields}), 
and \cite[Eq. (14) Ch.I.2]{Shields}
states that the first return time to $(-1,1)$ (\textit{i.e.}\  $T_2=\xi_1+\xi_2$)
has expectation $\bbP(\go_{1}=-1,\go_2=1)^{-1}<+\infty$.
\end{rem}

We now introduce a notion of ``good'' block: we call ``$A$-{\sl block}''
a segment $(T_i,T_{i+1}]$ on which $\go\equiv+1$ (take $i$ even), and whose size
$\xi_{i+1}$ is larger than $A$.
We define
\begin{equation}
\cT_1(A) = \inf \{ T_{2i+1} \ ; \ \xi_{2i+1}\geq A \}
\end{equation}
the position of the first $A$-{\sl block}, and iteratively,
$\cT_k(A)$ the position of the $k^{\rm th}$ $A$-{\sl block},
\begin{equation}
 \label{defTa}
\cT_{k}(A):= \inf \{ T_{2i+1}>\cT_{k-1}(A) \ ; \ \xi_{2i+1}\geq A \}.
\end{equation}
The quantities $\cT_k(A)-\cT_{k-1}(A)$ represent the distances
between the rare large attractive regions of length at least $A$,
and that is why the subexponential decay (in $A$)
of $\cT_k(A)-\cT_{k-1}(A)$ is used to characterize the presence of infinite disorder (recall Definition \ref{defstrong}).
We regroup in Figure \ref{fig:regions}
the above notations, \textit{i.e.}\  the decomposition of our environment $\go$
into elementary blocks $(T_{i-1},T_i]_{i\in\N}$ of size $\xi_i$, and for $A>0$ fixed, into meta-blocks
$(\cT_{k-1}(A),\cT_k(A)]_{k\in\N}$.
\begin{figure}[htbp]
\centerline{
\psfrag{0}{$0$}
\psfrag{xi1}{\footnotesize $\xi_1$}
\psfrag{xi2}{\footnotesize $\xi_2$}
\psfrag{xi3}{\footnotesize $\xi_3$}
\psfrag{t1}{\small $T_1$}
\psfrag{t2}{\small $T_2$}
\psfrag{t3}{\small $T_3$}
\psfrag{T1}{$\cT_1(A)$}
\psfrag{T2}{$\cT_2(A)$}
\psfrag{d1}{\small $\cT_1(A)$}
\psfrag{d2}{\small $\cT_2(A)-\cT_1(A)$}
\psfrag{d3}{}
\psfrag{go1}{$\go=-1$}
\psfrag{go0}{$\go=+1$}
\psfrag{ablock}{$A$-{\sl blocks}}
\psfrag{xii1}{\small $\xi_{k_1}\geq A$}
\psfrag{xii2}{\small $\xi_{k_2}\geq A$}
\psfig{file=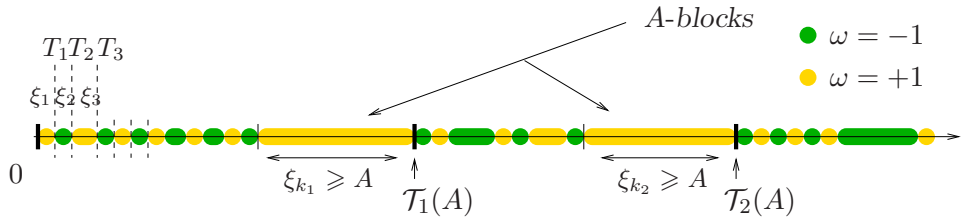,width=5in}}
  \begin{center}
    \caption{\label{fig:regions}
     Decomposition of the system into elementary segments $(T_{i-1},T_i]_{i\in\N}$ of size $\xi_i$,
     in which the value of $\go$ is constant. An $A$-{\sl block} is a segment $(T_{i-1},T_i]_{i\in\N}$
     constituted of $+1$'s, and which is larger than $A$.
     With a fixed parameter $A$, we divide our system into meta-blocks
     $(\cT_{k-1}(A),\cT_k(A)]_{k\in\N}$,
     composed of blocks with $\go\equiv-1$ or with length smaller than $A$, and then of one ending $A$-{\sl block}.
    }
  \end{center}
\end{figure}

We stress that the quantity
$\frac{1}{T_N} \sum_{n=1}^{N} \ind_{\{n \text{ is } odd\, ,\, \xi_n\geq A\}}$
represents the density of $A$-{\sl blocks} in a system of size $T_N$.
Using Birkhoff's Ergodic Theorem (cf. \cite[Chap. 2]{Nadk}), one gets that
the (asymptotic) density of $A$-{\sl blocks} is then equal to $\bbE[T_2]^{-1}\bbP(\xi_1\geq A)$.
Since one has
$\frac{\cT_N(A)}{N} = \left( \frac{1}{\cT_N(A)} \sum_{n=1}^{\gga_N(A)} \ind_{\{n \text{ is } odd\, ,\, \xi_n\geq A\}}\right)^{-1}$, where $\gga_N(A)$ is the index such that $T_{\gga_N(A)} = \cT_N(A)$
another use of the Ergodic Theorem yields that
the mean distance to the first $A$-{\sl block} is
\begin{equation}
\label{eq:T1A}
 \bbE[\cT_1(A)] =  
    \bbE[T_2]\,
  \bbP\left( \xi_1\geq A \right)^{-1} =
  \bbE[T_2] \bbP(\Fav_{\llbracket  1,A \rrbracket })^{-1}.
\end{equation}

\smallskip
In the sequel, constants whose precise value is not important will be in general denoted
$c,c',C,C'$, and to simplify notations their value can change from line to line.
We moreover keep track of
the dependence of the constants on $\gb$,
even if $\gb$ is thought as a fixed constant.

\subsection{Appearance of an infinite disorder regime}
\label{sec:mainthm}

The question we address here is that of understanding the threshold leading
from $h_c(\gb)>-\gb$ to $h_c(\gb)=-\gb$.
Our first theorem gives a criterion,
provided that correlations are non-increasing, in the following sense.
\begin{hyp}[Non-increasing correlations]
\label{hyp1}
For $l\in\bbN$ we use the notation $\theta^{l}\cE:= \{i+l, i\in\cE\}$.
The correlations between the events $\Fav_{\cE}$ 
are said to be non-increasing if, for every two sets of indices $\cE_1$ and $\cE_2$
verifying that $\max\cE_1 < \min\cE_2$, one has
$$\mathbb{P}(\Fav_{\cE_1\cup\, \theta^{k+1}\cE_2}) \leq \mathbb{P}(\Fav_{\cE_1\cup\, \theta^{k}\cE_2}) \quad \text{for all }  k\geq 0.$$
In other words, the covariances $\mathbb{C}\mathrm{ov} (\Fav_{\cE_1},\Fav_{\theta^k\cE_2})$
are non-increasing in $k$.
\end{hyp}

\begin{theorem}
\label{thm:criterion}
Under Assumption \ref{hyp1} 
one has the following criterion
\begin{equation*}
 h_c(\gb)=-\gb \ \textnormal{for all } \gb>0 \quad \Longleftrightarrow
   \quad \liminf_{n\to\infty} -\frac{1}{n} \log \bbP(\Fav_{\llbracket 1,n\rrbracket}) = 0
     \quad \text{(infinite disorder)}.
\end{equation*}
Moreover, in the conventional regime (\textit{i.e.}\ in absence of infinite disorder), one has that $h_c(\gb)>-\gb$ for all $\gb>0$.
\end{theorem}

We conjecture in Section \ref{sec:comments1} that Assumption \ref{hyp1}
is not necessary to get this criterion.
The next theorem already shows that certain implications of Theorem 
\ref{thm:criterion} hold under weaker assumptions.

\begin{theorem}
\label{thm:criticpoint}
Here we do not make  Assumption \ref{hyp1}.
\smallskip\\
{\rm (1)} Infinite disorder (Definition \ref{defstrong}) implies that $h_c(\gb)=-\gb$ for all $\gb> 0$.
\smallskip\\
{\rm (2)} If 
there exists a constant $c>0$ such that for all
sequence of indices $1\!\leq \! i_1\!<\!i_2\!<\!\cdots<i_n$
one has that $\bbP(\Fav_{\{i_1,\ldots,i_n\}})\leq e^{-c n}$ for all $n\geq 1$,
then there exists
a constant $c_{\gb}>0$ (uniformly bounded away from $0$
for $\gb\in(0,1)$) such that for all $\gb>0$ one has
$h_c(\gb)\geq -(1-c_{\gb}) \gb>-\gb$.
\end{theorem}
We point out that the condition in point (2) is weaker than assuming 
absence of infinite disorder 
plus Assumption \ref{hyp1} as was done in Theorem \ref{thm:criterion}. Indeed,
\begin{lemma}
\label{lemFKG}
Under Assumption \ref{hyp1}, if one has that
$\liminf_{n\to\infty } -\frac1n \log \bbP(\Fav_{\llbracket 1,n\rrbracket}) >0$,
then there exists a constant $c>0$ such that
for any sequence of indices $i_1\!<\! i_2\! <\! \cdots\! $
one has
\begin{equation}
\bbP(\Fav_{\{i_1,\ldots,i_n\}}) \leq e^{-cn}, \quad \text{for all } n\geq1.
\end{equation}
\end{lemma}

\begin{proof}
For all indices
$i_1\!<\!\cdots\!<\! i_n$, one has, thanks to repeated use of  Assumption \ref{hyp1},  that
$\bbP(\Fav_{\{i_1,\ldots,i_n\}})\leq \bbP(\Fav_{\llbracket 1, n\rrbracket})$,
thanks to Assumption \ref{hyp1}.
From this, one gets that
$\bbP(\Fav_{\{i_1,\ldots,i_n\}})\leq \bbP(\Fav_{\llbracket 1, n\rrbracket}) \leq e^{-c  n}$,
which gives the conclusion of Lemma \ref{lemFKG}.
\end{proof}

Point (1) follows from Proposition \ref{prop:boundsF} below.
Point (2) of Theorem \ref{thm:criticpoint} is proven 
 in Section \ref{sec:annclassic}, via bounds on the annealed model.

\subsection{Bounds on the free energy}
\label{sec:bounds}

\begin{definition}\rm
\label{defdefeps}
One defines
\begin{equation}
\label{defeps}
\begin{split}
 \gep(x):= \inf_{n\leq x} -\frac{1}{n}\log \bbP(\xi_1\geq n),
\end{split}
\end{equation}
which is non-increasing: one can therefore note $\gep^{-1}$,
its generalized inverse, defined by
$ \gep^{-1}(u):=\sup\{ x,\, \gep(x)\geq u \} $.
\end{definition}

We stress that $\bbP(\xi_1\geq n) \stackrel{n\to\infty}{=} o(1/n)$ since $\bbP(\xi_1\geq n)$ is
decreasing and
summable (because $\bbE[\xi_1]<\infty$).

\begin{proposition}
 \label{prop:boundsF}

{\bf 1. Lower bound.}
If $\lim_{x\to\infty} \gep(x) = 0$ (infinite disorder),
then $\gep^{-1}$ is defined on a neighborhood of $0$, and one has
$\lim_{u\to 0} \gep^{-1}(u)\! =\!+\infty$.
Then there exist two constants $c_0,c_0'>0$ (that do not depend on $\gb$), such that for all $u\in(0,1)$ and all $\gb>0$ one has
\begin{equation}
\label{lowboundF1}
 \tf(\gb,-\gb+u)\geq  c_0' A_u\bbP(\xi_1\geq A_u)\,  \tf(u) ,
\end{equation} 
where we defined $A_u:=\gep^{-1}(c_0 \tf(u))$, that goes to infinity
as $u$ goes to $0$.
In particular one has $h_c(\gb)=-\gb$, and the r.h.s. of \eqref{lowboundF1}
is $o(\tf(u))$ when $u$ goes to $0$ (since $A_u\bbP(\xi_1\geq A_u)$ goes to $0$).

\smallskip
{\bf 2. Upper bounds.} Whether $h_c(\gb)=-\gb$ or not, one also has
the following upper bound on the free energy:
there exist constants $C_1,c_1,\mathbf{c}>0$ (that do not depend on $\gb$),
such that for all $\gb\in(0,1) $and $u\in(0,c_1\gb)$ one has

\begin{equation}
\label{upboundF}
 \tf(\gb,-\gb+u)\leq
C_1  \gb^{1-\nu^{\rm pur}} \tf(u) \bbE\left[\frac{ T_{L(u)} }{L(u)}  
 \ind_{\left\{T_{L(u)}> 4\bbE[T_2]  L(u)\right\}}\right],
\end{equation}
where $L(u):=\lfloor\mathbf{c} \gb^{\nu^{\rm pur}} \tf(u)^{-1} \rfloor$.

One has another bound, easier to handle: for all $u\in(0,c_1\gb)$, one has
\begin{equation}
\label{simpleboundF}
 \tf(\gb,-\gb+u)\leq
C_1 u \bbE\left[\xi_1 \ind_{\left\{\xi_1>\mathbf{c} \gb u^{-1 }\right\}}\right].
\end{equation}

\end{proposition}

This Proposition gives simple and fruitful bounds and can
be  applied to many types of environment.
We give possible applications in the sequel, in particular
Theorem \ref{thm:signgauss}.
It is also used to prove Theorem \ref{thm:stronglyrel} (see Section \ref{sec:commentbounds}).

To prove the lower bound  (see Section \ref{sec:signlow}), we use a localization strategy. We focus on the contribution of the trajectories
that target an attractive region of size $A=A_u:= \gep^{-1}(c_0 \tf(u))$.
The energetic reward one gets on this region is $e^{\tf(u) A_u}$, and the entropic cost is
approximatively $\K\left(\bbP(\xi_1\geq A_u)\right)\approx e^{ - c A_u \gep(A_u)}$
(the first segment of size $A_u$ being at mean
distance $\bbP(\xi_1\geq A_u)^{-1}\approx e^{-A_u \gep(A_u)}$ from the origin).
Then our (optimized) choice of $A_u$ gives that $\tf(u) A_u > - c A_u \gep(A_u)$.
Targeting the first region where the energetic gain
overcomes the entropic cost of doing such a long jump
is therefore a good localization strategy.
The intuition that these are the only strategies contributing to the free energy
is confirmed by the 
(partly heuristic) reasoning of Section~\ref{sec:conj}, thanks to
a multiscale coarse-graining argument.
When using a simpler coarse-graining procedure,
we get the upper bound \eqref{upboundF}, see Section \ref{sec:upbound}.

\subsection{Case of the Gaussian signs environment}
\label{sec:envirsigngauss}

Let $\W:=(\W_n)_{n\geq 0}$ be a centered normalized stationary Gaussian process whose law is denoted $\bbP$, and
with correlation function $\rho_n:=\bbE[\go_i \go_{i+n}]$ (it does not depend on $i$ because of the stationarity), $\rho_0=1$.
We also assume that correlations are non-negative and are power-law decaying: there exist some $ a>0$ and some constant $c>0$ such that the correlation function $(\rho_n)_{n\geq0}$ verifies
\begin{equation}
 \rho_{k}\stackrel{k\to\infty}{ \sim} c  k^{- a}, \quad \text{ and } \rho_k\geq0 \text{ for all } k\geq 0.
\label{assumpgauss}
\end{equation}

It is natural from the Gaussian sequence $\{\W_n\}_{n\in\bbN}$, to define the environment
$\go$ with values in $\{-1,1\}$ by
\begin{equation}
 \label{envirsign}
\go_i:=\ind_{\{\W_i\geq0\}}- \ind_{\{\W_i< 0\}}.
\end{equation} 
We refer to this choice as the \emph{Gaussian signs} environment.
The sequence $\W$ is ergodic since $|\rho_k|\stackrel{k\to\infty}{\to} 0$ (see \cite[Ch.14 \S2, Th.2]{Cornfeld}). Therefore, $\go$ is also ergodic,
so that $\tf(\gb,h)$ exists, and one can apply Theorem \ref{thm:criterion} and Proposition \ref{prop:boundsF}.
We do that in Section \ref{sec:propgauss}, proving the following result.

\begin{theorem}
\label{thm:signgauss}
For the Gaussian signs environment defined above, with Assumption \eqref{assumpgauss}, one has:

$\bullet$ If $ a<1$, then $h_c(\gb)=-\gb$ for all $\gb>0$.
There exists some constant $c_2>0$, such that for all $u\in(0,1)$ and $\gb>0$ one has
\begin{equation}
\label{infFgauss}
\tf(\gb,-\gb+u)\geq \exp\left( - c_2  |\log u|^{1/(1- a)} \tf(u)^{- a/(1- a)}\right).
\end{equation}
There exist constants $c_3,c_3'>0$, such that
for all $\gb\in(0,1)$ and $u\in(0,c_3\gb)$ one has
\begin{equation}
\label{supFgauss}
 \tf(\gb,-\gb+u)\leq \exp\left( -c_3' u^{- a}  \right)
\end{equation} 

$\bullet$ If $ a> 1$, one has some $c_{\gb}>0$ such that
 $\tf^{\a}(\gb,h)\leq \tf(h+(1-c_{\gb})\gb)$ for all $\gb>0$, and in particular
$h_c(\gb)\geq h_c^{\a}(\gb) >-\gb$.

\smallskip
The previous bounds give a threshold leading from $h_c(\gb)>-\gb$ for $a>1$ to $h_c(\gb)=-\gb$
for $a<1$.
Moreover, if one has $\rho_{k+1}\leq \rho_k $ for all $k\geq 0$,
then one has the following criterion:
\begin{equation}
\label{eq:critergauss}
 h_c(\gb)=-\gb \ \  \Longleftrightarrow
  \ \  \liminf_{n\to\infty} -\frac{1}{n} \log \bbP(\W_1\geq 0, \dots, \W_n\geq 0) = 0 \quad \textrm{(infinite disorder)}.
\end{equation}
\end{theorem}

Note that in the Gaussian signs environment, if the correlation decay exponent is $a<1$, then the phase transition is of infinite order.
This stresses that disorder is strongly relevant in that case, in a substantial way.  It is the first example we are aware of in the pinning model framework where the presence of disorder makes the phase transition of infinite order (leaving aside the case $\ga=0$ where one already has that $\nu^{\rm pur}=\infty$).

\section{General comments on the results}
\label{sec:comments}

\subsection{Reduction to a \{-1,1\}-valued environment}
\label{sec:generalize}

Our choice of a $\{-1,1\}$-valued environment is made essentially to 
simplify notations, and to restrict the question on the critical point
to decide whether $h_c(\gb)=-\gb$ or $h_c(\gb)>-\gb$.
We now explain why this choice is actually not restrictive, and how general
environments can be treated with the
techniques used in this paper.

To a general bounded sequence $\go$,
we associate a binary sequence, that keeps track of
the distribution of "favorable" and "unfavorable" regions.
We note $\mathbf{M}:={\rm ess\, sup} (\go_1)$ ($\mathbf{M}<+\infty$),
we take $\eta$ small, and consider $I_{\eta}=[\mathbf{M}-\eta, \mathbf{M}]$ a neighborhood of $\mathbf{M}$. 
We now define the sequence
$w^{(\eta)}=(w_{i}^{(\eta)})_{i\in\bbN}\in\{-1,1\}^{\bbN}$ the following way:
\begin{equation}
w_i^{(\eta)}:= \ind_{\{\go\in I_{\eta}\}} -\ind_{\{\go_i\notin I_{\eta}\}}.
\end{equation}
Then all the results presented can be dealt via the sequence $w^{(\eta)}$,
studying its properties for arbitrary $\eta$.
For example, one can translate Theorem \ref{thm:criterion}
into the following assertion:

\smallskip
If the sequence $w^{(\eta)}$ verifies Assumption \ref{hyp1}
for every (small) $\eta>0$,
then one has
\begin{multline}
h_c(\gb)=-\mathbf{M}\gb \ \text{ for all } \gb>0 \\
 \Leftrightarrow \text{ for all } \eta>0,\ \liminf_{n\to\infty} -\frac{1}{n}
    \log \bbP\big( w_i^{(\eta)}=1 \ \text{ for all } 1\leq i\leq n\big) = 0.
\end{multline}

One also remarks that
$\go_n\leq \mathbf{M}-\eta/2 + \frac{\eta}{2} w_n^{(\eta)}$,
so that $Z_{N,h}^{\go,\gb} \leq Z_{N,h+\gb \mathbf{M}-\gb\eta/2}^{w^{(\eta)},\gb\eta/2}$.
The upper bounds one gets on the free energy in Proposition \ref{prop:boundsF}
can therefore also be generalized. In particular, one obtains that whenever $h_c(\gb)=-\mathbf{M}\gb$, disorder
is relevant for any value of $\nu^{\rm pur}$.
For the lower bounds, one simply uses the bound
$\go_n\geq (\mathbf{M}-\eta)\ind_{\{w_n^{\eta}=+1\}}$,
since only favorable regions are used in the proofs of our results.

\subsection{Discussion on the criterion in Theorem \ref{thm:criticpoint}}
\label{sec:comments1}

\subsubsection{Characterizations of infinite disorder}
\begin{lemma}
\label{lem:charactstrong}
The following conditions are equivalent, and all characterize infinite disorder: 

\vskip.05cm
{\rm (a)} \hskip.1cm $\liminf_{n\to\infty}-\frac{1}{n}\log \bbP(\Fav_{\llbracket 1,n\rrbracket})= 0$ ;
\vskip.05cm

{\rm (b)} \hskip.1cm  $\liminf_{n\to\infty}
 \frac{1}{n}\log \cT_1(n) =0$ $\bbP$-a.s. ;
 \vskip.05cm

{\rm (c)} \hskip.1cm 
  for all $\gd>0$, $A_{\gd}(\go)<+\infty$ $\bbP$-a.s.,
where we define for any $\gd>0$
\begin{equation}
\label{defAgo}
A_{\gd}^{(\go)} : = \inf \left\{A: \ \log \cT_1(A) <\gd A \right\}.
\end{equation}
\end{lemma}
\begin{proof}
Jensen's inequality gives that $\bbE[\log \cT_1(n)]\leq \log \bbE[\cT_1(n)]$,
which together with \eqref{eq:T1A} shows that {\rm (a)} implies
$\liminf_{n\to\infty}\bbE\left[
 \frac{1}{n}\log \cT_1(n)\right] =0$.
Then, a direct application of Fatou's Lemma
gives $\bbE\left[\liminf_{n\to\infty}
 \frac{1}{n}\log \cT_1(n)\right] =0$, and {\rm (b)}.
Using the definition of $A_{\gd}(\go)$, {\rm (b)} then directly implies that, for every fixed $\gd>0$,
$A_{\gd}(\go)<+\infty$ $\bbP$-a.s.

On the other hand, if $A_{\gd}(\go)<+\infty$ $\bbP$-a.s. for all $\gd>0$, then
it means that for any $p\in\bbN$,
there exists $\bbP$-a.s.\ some finite $A$ such that $\frac1A \log \cT_1(A)\leq 1/p$ .
It shows that $\liminf_{A\to\infty} \frac1A \log \cT_1(A) \leq 1/p$ $\bbP$-a.s.\ for any $p\in\bbN$,
which gives $\liminf_{A\to\infty} \frac1A \log \cT_1(A) =0 $  $\bbP$-a.s.

We are left to show that {\rm (b)} implies {\rm (a)}.
We actually prove a stronger statement:
\begin{equation}
\liminf_{n\to\infty}-\frac{1}{n}\log \bbP(\Fav_{\llbracket 1,n\rrbracket})>0
\ \ \Rightarrow \ \ 
\liminf_{n\to\infty}
 \frac{1}{n}\log \cT_1(n) >0 \ \ \bbP-a.s. 
\end{equation}
Indeed, if $ \bbP(\Fav_{\llbracket 1, n\rrbracket})=\bbP(\xi_1\geq n)\leq e^{-\gd n}$ for some $\gd>0$, then
one has that
\begin{equation}
\label{borelcant}
 \bbP(\cT_1(n)\leq  e^{\gd n/2})
   \leq \bbP \big( \exists k\leq  e^{\gd n/2} ,\ \xi_{2k+1}\geq n\big)\\
  \leq e^{\gd n/2} \bbP(\xi_1\geq n) \leq e^{-\gd n/2}, 
\end{equation}
where we first used that there are at most $\cT_1(n)$ segments $(T_{2k},T_{2k+1}]$
in a system of size $\cT_1(n)$, and then a union bound. From \eqref{borelcant} and using Borel-Cantelli Lemma,
one gets that $\bbP$-a.s., $\frac{1}{n} \log \cT_1(n) \leq \gd/2$ happens only
a finite number of time, meaning that $\liminf \frac{1}{n}\log \cT_1(n)>\gd/2$.
\end{proof}

\subsubsection{Conjecture on the sharp general criterion}

Let us now comment on the natural localization strategy that trajectories should
adopt when $h=-\gb+u$,
with $u$ small.
Considering the trajectories that aim directly at the first $A$-{\sl block},
one realizes that the energetic gain collected on this $A$-{\sl block} is
$A\tf(u)$, and that the
entropic cost of targeting it is $(1+\ga)\log \cT_1(A)$.
Then, looking back at the definition \eqref{defAgo}, one has that
 $\cT_1(A_{\tf(u)}^{(\go)}/(1+\ga))$ is the first time when the energetic gain overcomes the entropic cost.
It is therefore a natural conjecture to guess that,
if for some $u_0$ one has $A_{\tf(u_0)/(1+\ga)}^{(\go)}=+\infty$, then
the entropic cost will never be energetically compensated and one will be in the delocalized phase:
$h_c(\gb)\geq -\gb+u_0$.
We then can formulate a guess, that we justify in more details (but still to some extent heuristically) in Section \ref{sec:conj}.
\begin{conjecture}
\label{conj}
One has
the criterion
\begin{multline}
h_c(\gb)=-\gb  \ \text{for all } \gb>0\ \ \Leftrightarrow \ \
\text{for all }\gd>0, \ A_{\gd}(\go)<+\infty\ \bbP\textrm{-a.s.}\\
\text{(i.e. with infinite disorder, cf. Lemma \ref{lem:charactstrong})}.
\end{multline}
\end{conjecture}

\subsection{Comments on the bounds of Proposition \ref{prop:boundsF}}
\label{sec:commentbounds}

We first stress that Proposition \ref{prop:boundsF},
and in particular the upper bound \eqref{upboundF} on the free energy,
implies Theorem \ref{thm:stronglyrel}.
Indeed, thanks to the Ergodic Theorem,
one has 
that $\frac{T_{n}}{n} \stackrel{n\to\infty}{\to} \frac12 \bbE[T_2]$ $\bbP$-a.s.
It is then standard to obtain that
$\limsup_{n\to\infty} \bbE\left[\frac{T_{n}}{n} \ind_{\{T_n \geq 4\bbE[T_2]n \}}  \right]=0$,
which directly gives that $\tf(\gb,-\gb+u) \stackrel{u\searrow0}{=} o (\tf(u))$ from \eqref{upboundF}.
Moreover, in the case $\ga>1$, where $\tf(u)\stackrel{u\searrow0}{\sim} cst. u$,
the bound \eqref{simpleboundF} gives more directly
that $\tf(\gb,-\gb+u)\stackrel{u\searrow0}{=} o (\tf(u))$.

\smallskip
One is also able to get, in many cases,
a very simple bound on the free energy, knowing only
the behavior of $\bbP(\xi_1\geq A)$. Indeed,
a small computation gives that
\begin{equation}
 \bbE\left[ \xi_1 \ind_{\{\xi_1\geq A\}} \right] 
  = \sum_{n\geq A} n\left( \bbP(\xi_1\geq n) -\bbP(\xi_1\geq n+1) \right)
  = A \bbP(\xi\geq A) + \sum_{n> A}  \bbP(\xi_1\geq n).
\end{equation}
If $\bbP(\xi_1\geq A)$ decays
sufficiently fast (essentially faster than $A^{-(1+\gep)}$, one has to treat each case cautiously),
one gets
that $ \bbE\left[ \xi_1 \ind_{\{\xi_1\geq A\}} \right]\leq cA \bbP(\xi\geq A)$.
From \eqref{simpleboundF} one obtains
\begin{equation}
\label{roughF}
 \tf(\gb,-\gb+u)\leq
C' \bbP(\xi_1\geq c u^{-1}),
\end{equation} 
which gives a very explicit (but rough) upper bound.

\smallskip


Proposition \ref{prop:boundsF}
does not give optimal upper bounds on the free energy.
We actually believe that the lower bound \eqref{lowboundF1}
gives the right order for the critical behavior (up to some corrections, such as constants in the definition of $A_u$), based on the partially heuristic reasoning of Section \ref{sec:conj}.
We mention 
that Proposition \ref{prop:boundsF} 
is easily applicable when the sizes of the elementary blocks are independent, as in \cite{BLpin}.
One actually recovers the lower bound on the free energy of \cite[Thm. 2.1]{BLpin},
and the rough bounds given in \cite[Prop. 4.1 \& Prop. 5.1]{BLpin}.
The first step of the procedure proposed in Section \ref{sec:conj} (in particular \eqref{upFstep1}),
is however needed to
obtain the full result \cite[Thm. 2.1]{BLpin}.

\subsection{Properties of the Gaussian signs environment}
\label{sec:propgauss}

The following proposition, proven in Appendix \ref{app},
estimates the probability for a Gaussian correlated vector
to be componentwise non-negative (that corresponds to $\bbP(\xi_1\geq n)$
for the Gaussian signs environment).
Together with
Proposition \ref{prop:boundsF}, it gives \eqref{infFgauss}-\eqref{supFgauss}.

\begin{proposition}
\label{app:gauss}
We make the assumption \eqref{assumpgauss} on the Gaussian sequence $\W$.

$\bullet$ If $ a<1$, 
there exist two constants 
$c,c'>0$, such that for every $n\in\N$ one has
\begin{equation}
 \bbP\left( \W_i\geq 0 \, ; \, \forall i\in\{1,\dots,n\} \right)\geq e^{-c n^{ a} \log n}.
 \label{gausslow}
\end{equation}
Moreover, for all subsequences $1\leq i_1<\ldots<i_n$, one has
\begin{equation}
 \bbP\left( \W_i\geq 0 \, ; \, \forall i\in\{i_1,\dots,i_n\} \right)\leq  e^{ -c' n^{ a}}.
 \label{gaussup1}
\end{equation}

$\bullet$ If $ a>1$, there exists some constant $c''>0$ such that
for all subsequences $1\leq i_1<\ldots<i_n$ one has
\begin{equation}
 \bbP\left( \W_i\geq 0 \, ; \, \forall i\in\{i_1,\dots,i_n\} \right)\leq  e^{ -c'' n}.
 \label{gaussup2}
\end{equation}
\end{proposition}

Let us mention that in the case $ a<1$,
for a rather particular choice of the Gaussian covariance structure,
  \cite[Th.1.1]{BDZ95} gives a much sharper result than ours.
More precisely, in the case where the covariances of $(\W_n)_{n\in\N}$ are given by the Green function
of some transient random walk on $\Z^d$ (one can
construct such a random walk in a way that $\rho_n\sim c_{ a} n^{- a}$,
with some explicit constant $c_{ a}$, see \cite{BDZ95}), one has

\begin{equation}
\label{app:gaussprecis}
 \lim_{n\to\infty} -\frac{1}{n^{a} \log n }
\log \bbP(\W_i\geq 0 \, ; \, \forall i\in\{1,\dots,n\}) = \mathsf{C}_{ a},
\end{equation} 
where the constant $\mathsf{C}_{ a}$ is explicit.
If we had an estimate like \eqref{app:gaussprecis} in Proposition \ref{app:gauss}
for $a<1$, 
we would 
get a slightly more precise upper bound in Theorem \ref{thm:signgauss}.
As we do not hunt for the sharp behavior in Theorem \ref{thm:signgauss},
we are satisfied with Proposition \ref{app:gauss} which is valid
with very little assumptions on the correlation structure.

\begin{rem}\rm
 \label{rem:xi1}
The case $ a=1$ is more problematic because it is a marginal case,
and our proofs would
adapt to this case, giving
\begin{equation}
e^{-cn/\log n}\geq \bbP\left( \W_i\geq 0 \, ; \, \forall i\in\{1,\dots,n\} \right) \geq e^{-cn}.
\end{equation}
In view of \eqref{app:gaussprecis},
and because the term $n^{ a}$ ($=\sum_{k=1}^n \rho_k$) when $ a<1$
would be replaced by $n/\log n$ if $ a=1$,
we believe that $\log \bbP\left( \W_i\geq 0 \, ; \, \forall i\in\{1,\dots,n\} \right)$
is of order $n$. When $ a=1$,
one would therefore have
a statement similar to the case $ a>1$.
The system should therefore stand in the {\sl conventional} regime where
$h_c(\gb)>-\gb$ for small $\gb$.
\end{rem}

\begin{proof}[Proof of Theorem \ref{thm:signgauss}]

If $ a<1$, Proposition \ref{app:gauss} gives that $\bbP(\xi_1\geq n)\geq e^{-c_2 n^{ a}\log n}$.
One has $ \gep(x)\leq c' x^{ a-1} \log x$, and thus there exists
a constant $c$ such that 
$\gep^{-1}(t) \leq c   (t^{-1}|\log t|)^{1/(1- a)}$.
Defining $A_u := \gep^{-1}(c_0 \tf(u))$ as in Proposition \ref{prop:boundsF},
one gets $A_u\leq c (\tf(u)^{-1}|\log u|)^{1/(1- a)}$ (recall that $\tf(u)$ is of polynomial order).
One concludes using again Proposition \ref{app:gauss}, which
gives that
\begin{equation}
\bbP(\xi_1\geq A_u) \geq \exp\left( - c |\log u|^{1/(1- a)} \tf(u)^{- a/(1- a)}  \right),
\end{equation}
that combined with \eqref{lowboundF1} brings \eqref{infFgauss}.
Moreover, Proposition \ref{app:gauss}
also implies that $\bbP\left( \xi_1\geq n\right)\leq e^{-c'_2 n^{ a}}$
so that \eqref{supFgauss} follows directly from the bound \eqref{roughF}.
In the case $ a>1$, the conclusion follows from Proposition \ref{app:gauss}, that gives the condition $(2)$ in Theorem \ref{thm:criticpoint}.

To get the criterion \eqref{eq:critergauss} we have to show that the
Gaussian signs environment satisfies Assumption \ref{hyp1} if the correlation function is non-increasing, to then be able to use Theorem \ref{thm:criterion}.
Indeed, let us consider two sets of indices $\cE_1$
and $\cE_2$, with $\max \cE_1 < \min \cE_2$.
We notice that, if $\rho_{k+1}\leq \rho_k$ for all $k\geq 0$, then the covariances
of the Gaussian vector $X=(\W_i)_{i\in \cE_1\cup\, \theta^{k+1} \cE_2}$ are all smaller
or equal than those of 
$ Y=(\W_i)_{i\in \cE_1\cup\, \theta^{k} \cE_2}$, for any $k$.
Then one uses a convenient equality due to Piterbarg \cite{Pit82}.
\begin{proposition}
Let $X = (X_1, . . . , X_n)$ and $Y = (Y_1, . . . , Y_n)$ be two independent families of centered Gaussian random variables. Let $g : \bbR^n \to\bbR^1$ be a function with bounded second derivatives. Then
\[ \bbE[g(X)]-\bbE[g(Y)] =
\frac{1}{2} \sum_{i,j=1}^n \big( \bbE[X_iX_j] -\bbE[Y_iY_j] \big) \int_{0}^1
\frac{\partial^2 g}{\partial x_i \partial x_j} \big( (1-t)^{1/2}X+t^{1/2} Y \big) \dd t.
\]
\end{proposition}
We approximate the indicator function $\ind_{\{x\geq 0\}}$
by a twice differentiable non-decreasing function $f$,
and $\ind_{\{x_i\geq 0,\, \forall i\in \{1,\ldots,n\}\}}$ by
$g(x_1,\ldots,x_n)=\prod_{i=1}^n f(x_i)$, we have in particular that
$\frac{\partial^2 g}{\partial x_i \partial x_j}  \geq 0$
for every $i\neq j$.
Since in our case
$\bbE[X_iX_j] \leq \bbE[Y_iY_j]$ for all $i,j\in\{1,\ldots,n\}$ (with equality for $i=j$),
we get that $\bbE[g(X)]\leq \bbE[g(Y)]$.
In the end, one obtains that
$\mathbb{P}(\Fav_{\cE_1\cup\, \theta^{k}\cE_2}) \geq \mathbb{P}(\Fav_{\cE_1\cup\, \theta^{k+1}\cE_2})$ for all $k\geq 0$.
\end{proof}

We stress that one is also able to estimate the covariances of the Gaussian signs environment:
they have the same decay exponent as the correlation function $\rho_n$:
thanks to a standard Gaussian computation, one gets that
$
\mathbb{C}\mathrm{ov}(\go_i,\go_{i+k}) \stackrel{k\to\infty}{\sim} \frac{2\rho_k}{\pi}
$.

\subsection{Conventional vs.\ Infinite disorder regime}

We now stress the characteristics of the two regimes, to confront their respective properties.
In particular, we collect the main results on the appearance of
the infinite disorder regime in Figure \ref{fig:threshold},
focusing on the
Gaussian signs environment
to make exposition clearer.

\begin{figure}[htbp]
\centerline{
\psfrag{0}{$-\gb$}
\psfrag{h}{$h$}
\psfrag{f}{$\tf(\gb,h)$}
\psfrag{fa}{\footnotesize $\tf^{\a}(\gb,h)$}
\psfrag{fapur}{\footnotesize $\tf^{\a}(\gb,h)\!\!= \!\!\tf(0,h+\gb)$}
\psfrag{compann}{\small $\nu^{\a}=1\wedge1/\ga$}
\psfrag{compq}{\small $\nu^{\q}=\infty$}
\psfrag{degenerated}{{\color{red} \footnotesize {\sl infinite disorder} regime} }
\psfrag{nondegenerated}{ {\color{red} \footnotesize {\sl conventional} regime}}
\psfrag{hc}{$h_c(\gb)$}
\psfrag{hca}{$h_c^{\a}(\gb)$}
\psfrag{relev}{\footnotesize ? Weinrib-Halperin criterion ?}
\psfig{file=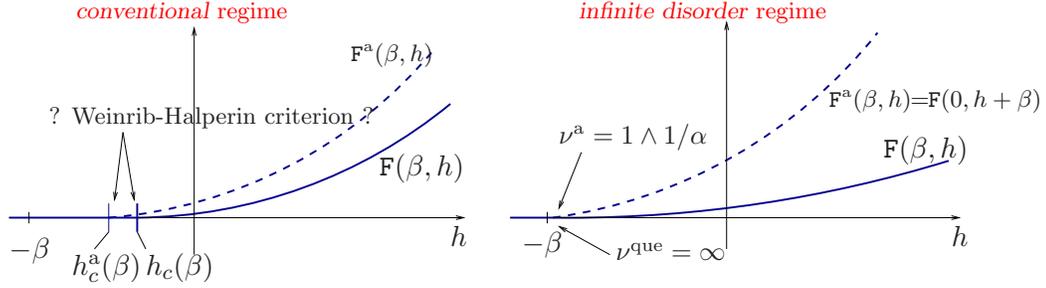,width=5in}}
  \begin{center}
    \caption{ \label{fig:threshold}
    Conventional vs. Infinite disorder regime
    in the case of the Gaussian signs environment.
      In the conventional regime
      (which happens if $a\!>\!1$),
     one has that  $h_c(\gb)\!\geq\! h_c^{\a}(\gb)\!>\! -\gb$.
     The question of disorder relevance/irrelevance is still not settled,
     the Harris criterion is believed to be valid in this regime.
     In the infinite disorder regime
      (which happens if $a\!<\!1$), the picture is unconventional:
      the annealed model is trivial, the quenched
      critical point $h_c(\gb)$ is equal to its minimal possible value $-\gb$,
      and the phase transition is of infinite order.
 }
  \end{center}
\end{figure}

\textbullet\ Conventional regime:
$\bbP(\xi_1\geq n)$ decays exponentially fast.
In terms of the distribution
and size of favorable segments in the environment,
one has the same typical properties as in the IID case
(\textit{i.e.}\ if $-1$'s and $+1$'s were IID): favorable regions
are (exponentially) far one from another and do not aggregate.
We hence expect that in this case the ideas of Harris and Weinrib-Halperin are applicable,
giving a criterion
for disorder relevance/irrelevance.

\textbullet\ Infinite disorder regime:
$\bbP(\xi_1\geq n)$ decays subexponentially.
It means that large favorable regions are much more frequent (or
aggregate much more) than in the IID case, and are in particular
only subexponentially  far one from another.
The pinning model then exhibits some unusual behavior:
the critical point is equal to its minimal possible value ($h_c(\gb)=-\gb$),
the annealed model is trivial (one has $\tf^{\a}(\gb,h)=\tf(h+\gb)$, as
if the environment were constituted of only $+1$'s, see Section \ref{sec:annealed}), and
disorder is strongly relevant.
Moreover, if $\bbP(\xi_1\geq n)$ decays faster than any power of $n$
(for example stretched-exponentially), then the bound \eqref{roughF} yields
that the transition is actually of infinite order.

\begin{rem}\rm
In view of the localization strategy described in Section \ref{sec:bounds},
one realizes that the farther
the first large attractive region is (\textit{i.e.}\
the smaller $\bbP(\xi_1\geq n)$ is), the lower the contact fraction should be.
In \eqref{roughF}, and in \eqref{supFgauss} ,
one verifies this fact:
the faster $\bbP(\xi_1\geq n)$ decays,
the smoother the free energy critical behavior is.
If the infinite disorder regime is barely reached,
it means that attractive regions are
"almost" exponentially far one from another, and that
the free energy should vanish "almost" exponentially fast as $h$ approaches $-\gb$.
On the other hand, if the environment possess extremely large and frequent favorable regions,
then the critical behavior of the disordered model is becoming closer to
that of the pure one (see also Theorem 1.5 in \cite{BLpin}).
\end{rem}

\section{Annealed estimates, proof of Proposition \ref{prop:ann}}
\label{sec:annealed}

We now comment on the annealed model,
that exhibits a trivial behavior in the infinite disorder regime.
\begin{proposition}
\label{prop:ann}
 Under Assumption \ref{hyp1}, one has
 the following criterion
\begin{equation}
 \liminf_{n\to\infty} -\frac{1}{n} \log \bbP(\Fav_{\llbracket 1,n\rrbracket}) = 0 \quad \Longleftrightarrow 
   \quad \tf^{\a}(\gb,h)=\tf(0,h+\gb) \ \text{for all } h\in\bbR, \gb\geq0 .
\end{equation}
\end{proposition}

\subsection{Triviality of the annealed system in the infinite disorder regime}

If one has that $\liminf_{n\to\infty} - \frac1n \log \bbP(\xi_1\geq n)=0$
(infinite disorder), the average behavior of the partition function is actually dominated
by the very large fluctuations in the disorder: the annealed model is trivial.
Indeed, imposing all $\go_i$'s to be equal to $+1$ in a system of size $N$, one obtains the bound
$\Eo Z_{N,h}^{\go,\gb} \geq \Po(\xi_1\geq N) Z_{N,h+\gb}^{\rm pur}$. Thus, we readily have
\begin{equation}
\label{trivann}
 \tf^{\a}(\gb,h)= \lim_{N\to\infty} \frac1N \log \Eo Z_{N,h}^{\go,\gb}
   \geq \limsup_{N\to\infty}\frac1N \log \Po(\xi_1\geq N)  +\tf(0,h+\gb) = \tf(0,h+\gb).
\end{equation}
Since one has the other trivial bound $Z_{n,h}^{\go,\gb}\leq Z_{n,h+\gb}^{\rm pur}$,
one gets $\tf^{\a}(\gb,h) = \tf(0,h+\gb)$, which gives the first
part of Proposition \ref{prop:ann}
(the other implication is proven by Proposition~\ref{prop:localann}).

The bound
$\tf(\gb,h)\leq \tf^{\a}(\gb,h)$ therefore gives no more information
than the trivial one $\tf(\gb,h)\leq \tf(0,h+\gb)$ (see Figure \ref{fig:threshold}).


\smallskip
{\bf On the correlation lengths.}
In the IID  case, $\tf(\gb,h)^{-1}$ is $\bbP$-a.s.\ equal to the
exponential decay rate of the two-point correlation function
$\E_{N,h}^{\go,\gb}(\gd_i\gd_{i+k})-\E_{N,h}^{\go,\gb}(\gd_i)\E_{N,h}^{\go,\gb}(\gd_{i+k})$
when $k\to\infty$, 
as proven in \cite[Th.3.5]{T06}
(when $\bP$ is the law of the return times to the origin of the simple random walk).
Therefore, $\tf(\gb,h)^{-1}$ is the \textit{quenched} correlation length in the
IID  case.
One also defines the usual \textit{quenched-averaged} correlation length, \textit{i.e.}\
the inverse of the exponential decay rate of 
$\bbE[\E_{N,h}^{\go,\gb}(\gd_i\gd_{i+k})-\E_{N,h}^{\go,\gb}(\gd_i)\E_{N,h}^{\go,\gb}(\gd_{i+k})]$.
In the IID case, it
is shown to be equal to $\mu(\gb,h)^{-1}$ \cite[Th.3.5]{T06}
(under the particular assumptions already mentioned),
where
\begin{equation}
\label{defmu}
 \mu(\gb,h):=-\lim_{N\to\infty} \frac{1}{N}\log\bbE\Big[ \frac{1}{Z_{N,h}^{\go,\gb}}\Big].
\end{equation}

We believe that
this correlation length(s) interpretation is still valid in the correlated framework.
One easily gets
from Jensen inequality that $\mu(\gb,h)\leq \tf(\gb,h)$,
and in the IID  framework
one actually has that $c_{\gb}\tf(\gb,h)^2<\mu(\gb,h)<\tf(\gb,h)$ for $h>h_c(\gb)$
(a better lower bound is given in \cite[Th.3.3]{T06}).
It means that
the {\sl quenched} and {\sl quenched-averaged} correlation lengths diverge at the
same critical point, namely $h_c(\gb)$.

We now stress that this picture is fully changed in the infinite disorder regime.
If we assume that the distribution of $+1$'s and $-1$'s are symmetric, then
one also has that $\liminf_{n\to\infty} -\frac1n \log \bbP(\xi_2\geq n)=0$,
and the average in \eqref{defmu}
is dominated by the rare repulsive regions.
One has
\begin{equation}
\frac1N \log \bbE\Big[ \frac{1}{Z_{N,h}^{\go,\gb}}\Big]
 \geq \frac1N\log \bbP(\xi_1=1,\xi_2\geq N)  -\frac1N\log Z_{N,h-\gb}^{\rm pur},
\end{equation} 
which directly give that $\mu(\gb,h)\leq \tf(0,h-\gb)$ by letting $N$ go to infinity.
The trivial bound $Z_{N,h}^{\go,\gb}\geq Z_{N,h-\gb}^{\rm pur}$
then gives that $\mu(\gb,h)=\tf(0,h-\gb)$.

Therefore, in the infinite disorder regime,
the {\sl quenched} correlation length $\tf(\gb,h)^{-1}$ diverges at $h_c(\gb)=-\gb$
with a critical exponent larger than $\nu^{\rm pur}$
and possibly infinite (see Theorems \ref{thm:stronglyrel}-\ref{thm:signgauss}).
On the other hand. the {\sl quenched-average} correlation length
$\mu(\gb,h)^{-1}=\tf(0,h-\gb)^{-1}$ diverges at $+\gb$, with critical exponent $\nu^{\rm pur}$.
This stresses again the unconventional behavior
of the system in presence of infinite disorder,
where $\tf(\gb,h)$ and $\mu(\gb,h)$
have different critical points, and where $\tf(\gb,h)$
has a larger critical exponent than $\mu(\gb,h)$.

\subsection{Annealed bounds in the conventional regime}
\label{sec:annclassic}

The following result shows that in the {\sl conventional} regime, the annealed
bound $\tf(\gb,h)\leq \tf^{\a}(\gb,h)$ is not trivial:
bounds on the annealed free energy are fruitful,
in particular to show that $h_c(\gb)>-\gb$. 
\begin{proposition}
 \label{prop:localann}
If there exists a constant $c_0>0$ such that
for all indices $1\leq i_1<\ldots<i_m$ one has
$\bbP\left( \Fav_{\{i_1,\ldots, i_n\}} \right)\leq e^{-c_0 n}$,  
then 
for all $\gb>0$ one has
\begin{equation}
\tf^{\a}(\gb,h) \leq \tf(0,h+(1-c_{\gb})\gb) \quad 
\text{with } c_{\gb}:=-\frac{1}{\gb} \log \big( 1\!-\! (1\!-\!e^{-c_0})(1\!-\!e^{-\gb}) \big)>0.
\end{equation}
In particular $h_c(\gb)\geq h_c^{\a}(\gb)\geq -(1-c_{\gb})\gb>-\gb$ for all $\gb>0$.
Note that the constant $c_{\gb}$ is uniformly bounded away from $0$ for $\gb\in(0,1)$.
\end{proposition}

This result gives
the second part of Proposition \ref{prop:ann}
(since the condition in Proposition \ref{prop:localann} is stronger than Assumption \ref{hyp1}, recall Lemma \ref{lemFKG}) and of Theorem \ref{thm:criticpoint}.
\begin{proof}
First, we use a simple bound $\go_n\leq \tilde \go_n := \go_n\ind_{\{\go_n=+1\}}$,
so that one has $Z_{N,h}^{\go,\gb}\leq Z_{N,h}^{\tilde\go,\gb}$.
Then, as $\tilde\go_n\in\{0,1\}$,
we can expand $e^{\gb\sum_{n=1}^{N}\tilde\go_n \gd_n}$,
thanks to the following binomial expansion
\begin{equation}
 e^{\gb \sum_{n=1}^{N}\tilde\go_n\gd_n } =  (e^{\gb}-1+1)^{\sum_{n=1}^{N}\tilde \go_n\gd_n}= \sum_{m=0}^{N} (e^{\gb}-1)^{m}
   \sum_{1\leq i_1<\ldots<i_m\leq N} \prod_{k=1}^{m}\tilde\go_{i_k} \gd_{i_k}.
\end{equation}
Thus, using the condition $\bbP\left( \Fav_{\{i_1,\ldots, i_n\}} \right)\leq e^{-c_0 n}$, one gets
\begin{multline}
\Eo\left[ e^{\gb\sum_{n=1}^{N}\tilde\go_n\gd_n}\right]
  = \sum_{m=0}^{N} (e^{\gb}-1)^{m}
   \sum_{1\leq i_1<\ldots<i_m\leq N}
        \Po\left( \go_{i_1}=1,\ldots,\go_{i_m}=1 \right)\prod_{k=1}^{m}\gd_{i_k} \\
 \leq  \sum_{m=0}^{N} (e^{\gb}-1)^{m}e^{-c_0 m}
   \sum_{1\leq i_1<\ldots<i_m\leq N} \prod_{k=1}^{m}\gd_{i_k} = (e^{-c_0}(e^{\gb}-1)+1)^{\sum_{n=1}^{N}\gd_n}.
\end{multline}
Defining $c_{\gb}:=-\frac{1}{\gb} \log \scriptstyle{\left( 1- (1-e^{-c_0})(1-e^{-\gb}) \right)}$
such that $e^{-c_0}(e^{\gb}-1)+1 =  e^{(1-c_{\gb}) \gb}$, one finally obtains that 
for all $\gb>0$
\begin{equation}
 \Eo Z_{N,h}^{\go,\gb}
  \leq  \bE\left[ e^{(h+(1-c_{\gb})\gb)\sum_{n=1}^{N}\gd_n}\right]
    = Z_{N,h+(1-c_{\gb})\gb}^{\rm pur},
\end{equation} 
which gives the results.
\end{proof}

\section{Bounds on the free energy, proof of Proposition \ref{prop:boundsF}}
\label{sec:proof}

\subsection{Lower Bound on the free energy}
\label{sec:signlow}

First of all, let us define the \emph{pinned} partition function, where trajectories are constrained to return
to $0$ at their endpoint $N$:
$ Z_{N,h}^{\go,\gb,\pin}:= \bE\left[ \exp\left( \sum_{n=1}^N  (h+\gb\go_n) \gd_n\right)\ind_{\{N\in\tau\}} \right]$.
The main advantage of the \emph{pinned} partition function is that
it has a (ergodic) supermultiplicative property:
\begin{equation}
Z_{N+M,h}^{\go,\gb,\pin} \geq \bE\left[ e^{\sum_{n=1}^N  (h+\gb\go_n) \gd_n} 
      \ind_{\{N\in\tau\}} \ind_{\{N+M\in\tau\}}\right]
           = Z_{N,h}^{\go,\gb,\pin} \times Z_{M,h}^{\theta^N \go,\gb,\pin},
\end{equation}
where $\theta$ is the shift operator: $\theta^{p} \go = (\go_{i+p})_{i\in\bbN}$.
This translate into a super-additive property for the $\log$ of the pinned partition function.
We also introduce the notation used in \cite{BLpin}
of the partition functions (with free or pinned boundary condition) over
a given segment $[a,b]$, $a,b\in\N$, $a<b$:
\begin{equation}
\label{defZblock}
Z^{\go,\gb}_{[a,b],h}:= e^{\gb\go_a+h}Z_{(b-a),h}^{\theta^a \go,\gb},\quad \text{and} \quad
Z^{\go,\gb,\pin}_{[a,b],h}:=Z_{(b-a),h}^{\theta^a \go,\gb,\pin}.
\end{equation}
We mention that as $\go$ is ergodic (thus stationary), $Z_{(b-a),h}^{\theta^a \go,\gb}$
has the same law as $Z_{(b-a),h}^{\go,\gb}$.
With notations \eqref{defZblock}, the super-additive property of the $\log$ of the pinned partition function extends to
partition functions over segments: for any non-negative integers $a<b<c$, one has
\begin{equation}
\label{superad}
\log Z_{[a,c],h}^{\go,\gb,\pin} \geq \log Z_{[a,b],h}^{\go,\gb,\pin} +\log Z_{[b,c],h}^{\go,\gb,\pin}.
\end{equation}

\subsubsection{General lower bound}

We work with the pinned partition function,
and we recall that, as far as the free energy is concerned, this is equivalent to working with
the partition function with ``free'' boundary condition (see \cite[Rem. 1.2]{GBbook}).
To get a lower bound on the free energy, we use a classical technique, that is
to find a strategy of localization for the polymer, aiming only at some particular blocks.
We give a very general lower bound, valid for any strategy, where trajectories are "aiming" at particularly favorable stretches in the environment.
We divide the environment in a sequence of blocks
$(\mathbf{T}_j, \mathbf{T}_{j+1}]_{j\geq 0}$
(we take for example $\mathbf{T}_j$ for among the $T_i$'s).
The super-additivity of the logarithm of the partition function (see \eqref{superad}) gives
\begin{equation}
\label{firstlowbound}
 \log  Z_{\mathbf{T}_N,h}^{\go,\gb,\pin}
 \geq \sum_{j=1}^{N} \log Z_{[\mathbf{T}_{j-1}, \mathbf{T}_{j}],h}^{\go,\gb,\pin},
\end{equation}
where we used notation \eqref{defZblock}, and that $\mathbf{T}_0=0$

One natural choice is actually to define the sequence $(\mathbf{T}_j)_{j\geq0}$ iteratively: $\mathbf{T_0}=0$, $\mathbf{T}_1=\mathbf{T}_1^{h,\gb}(\go)$ is the first time when a "favorable" region ends (for example, $\cT_1(A)$ for some $A$), and then for any $k\geq1$, define $\mathbf{T}_{k+1}=\mathbf{T}_1\left(\theta^{\mathbf{T}_k}\go\right)$. Then, the increments $(\mathbf{T}_{k}-\mathbf{T}_{k-1})_{k\geq 1}$ form an ergodic sequence \cite[Th.I.2.19]{Shields} , and have the same law as $\mathbf{T}_{1}$ ($\mathbf{T}_k$ can be thought as \emph{return times} in terms of the sequences $(\go_n)_{n\geq 0}$ or $(T_n-T_{n-1})_{n\geq 1}$).

For instance, one can take
$\mathbf{T}_k=T_{kN_0}$ for some fixed $N_0$, or
$\cT_k(A)$ for some (deterministic or random) number $A$.
This iterative definition allows us to choose any initial strategy (finding the first favorable region),
and then to repeat this strategy all along the environment.
In that case, and provided that $\bbE[\mathbf{T}_1]<\infty$,
 \eqref{firstlowbound} gives
\begin{multline}
\label{lowgeneralF}
\tf(\gb,h) = \lim_{N\to\infty} \frac{1}{\mathbf{T}_N} \log  Z_{\mathbf{T}_N,h}^{\go,\gb,\pin}\\
   \geq \lim_{N\to\infty} \frac{N}{\mathbf{T}_N} \frac1N
      \sum_{j=1}^{N} \log Z_{[\mathbf{T}_{j-1}, \mathbf{T}_{j}],h}^{\go,\gb,\pin}
        = \frac{1}{\bbE[\mathbf{T}_1]} \bbE\left[ \log Z_{\mathbf{T}_1,h}^{\go,\gb,\pin }\right],
\end{multline}
where we used twice Birkhoff's Ergodic Theorem for the last equality.

\smallskip
From now on, we write $h:=-\gb+u$,
since we study the behavior of the free energy as $h$
goes to $-\gb$ ($\gb$ being a fixed constant).
Thanks to the bound \eqref{lowgeneralF}, one is only left with picking some 
$\mathbf{T}_1$ wisely, and building a 
strategy of localization 
on the segment $[0,\mathbf{T}_1]$.
We take
$\mathbf{T}_1$ the first position where some large attractive region appears,
and target directly this region.
Namely, one chooses
$\mathbf{T}_1=\cT_1(\mathbf{A})$ for some $\mathbf{A}:=\mathbf{A}(\gb,u,\go)$ (that has yet to be chosen),
and target directly the last $\mathbf{A}$-{\sl block}, where $\go\equiv+1$ 
(recall notations of Section \ref{sec:notations}).
This leads to the lower bound
$Z_{\cT_1(\mathbf{A}),-\gb+u}^{\go,\gb,\pin}\geq \K(\cT_1(\mathbf{A}) - \mathbf{A}) Z_{\mathbf{A}, u }^{\rm pur,\pin}$.

Now, one uses \cite[Lem. 3.1]{BLpin} that gives $Z_{n,u}^{\rm pur,\pin}\geq C n^{-1} e^{n\tf(u)}$
for all $u\in(0,1)$
and all $n\in\bbN$.
Together with the assumption on $\K(\cdot)$ that yields that there is a constant $c$ such that
$\K(n)\geq c n^{-(1+\ga)}$, one finally gets
\begin{equation}
 \log Z_{\cT_1(\mathbf{A}),-\gb+u}^{\go,\gb,\pin} 
   \geq -(1+\ga) \log \cT_1(\mathbf{A}) + \mathbf{A}\tf(u)- \log \mathbf{A} -C.
\end{equation}
Then, using that $\cT_1(\mathbf{A})\geq \mathbf{A}$ (and $\mathbf{A}$ is also chosen larger than $2$)
and recalling \eqref{lowgeneralF},
one ends up with
\begin{equation}
\label{goodbound}
\tf(\gb,-\gb+u) \geq \frac{1}{\bbE[\cT_1(\mathbf{A})]}
\big( \bbE\left[\mathbf{A}\right]\tf(u) - \mathbf{C} \bbE[\log \cT_1(\mathbf{A})]\big), \ \ \ \text{for all } u\in(0,1),
\end{equation}
with $\mathbf{C}$ a given constant depending only on the law of the renewal. The choice of the strategy is now reduced to the choice of $\mathbf{A}=\mathbf{A}(u,\go)$ (we do not have to optimize on $\gb$ since none of the constants depend on $\gb$).
To get localization, \textit{i.e.}\ $\tf(\gb,h)>0$,
$\mathbf{A}$ must be such that $\bbE[\cT_1(\mathbf{A})]$ is finite, and also such that
$\bbE\left[\mathbf{A}\right]\tf(u) > \mathbf{C} \bbE[\log \cT_1(\mathbf{A})]$.

\subsubsection{Choice of the strategy}

We now 
take $\mathbf{A}:=A_u$ non-random, chosen in a moment.
One gets from \eqref{goodbound} that 
\begin{multline}
\tf(\gb,-\gb+u) \geq \frac{1}{\bbE[\cT_1(A_u)]}
A_u \left( \tf(u) - \mathbf{C}\frac{1}{A_u}\bbE[\log \cT_1(A_u)]\right) \\
  \geq   cst.\,  A_u\, \bbP(\xi_1\geq A_u) \, \left(  \tf(u) - \mathbf{C} (-\frac{1}{A_u}  \log \bbP(\xi_1\geq A_u) ) \right) \label{eq1},
\end{multline}
where we used Jensen's inequality to have that $\bbE[\log \cT_1(A_u)]\leq \log \bbE[\cT_1(A_u)]$,
equation \eqref{eq:T1A} to get $\bbE[\cT_1(A_u)]=cst. \bbP(\xi_1\geq A_u)^{-1}\geq \bbP(\xi_1\geq A_u)^{-1}$.
and the definition \eqref{defeps} of $\gep$.
Then, in \eqref{eq1} one chooses $A_u:= \gep^{-1}\left( \frac{\tf(u)}{2\mathbf{C}} \right)$, to finally get \eqref{lowboundF1}.

\begin{rem}\rm
We now comment on the natural strategy, sketched in Section \ref{sec:comments1}:
one would choose $\cT_1(\mathbf{A})$ the first position
where the energetic gain on the $\mathbf{A}$-\textsl{block}, $\mathbf{A}\tf(u)$, compensates the entropic cost of targeting it,
$-\mathbf{C}\log \cT_1(\mathbf{A}$.
Recall the notation \eqref{defAgo}: for every $\gd>0$
$A_{\gd}^{(\go)} : = \inf \left\{A, \log \cT_1(A) <\gd A \right\}$.
It is then natural to take $\mathbf{A}:=A_{\tf(u)/\mathbf{C}}^{(\go)}$ (which is random),
since
$A_{\tf(u)/\mathbf{C}}^{(\go)}\tf(u)-\mathbf{C}\log \cT_1(A_{\tf(u)/\mathbf{C}}^{(\go)})> 0$.
One has however to be careful to show that this is also a good localization strategy,
because the condition $\bbE[\cT_1(A_{\tf(u)/\mathbf{C}})]<\infty$ in \eqref{goodbound} is not guaranteed here.
\end{rem}

\subsection{Upper bound on the free energy}
\label{sec:upgauss}

In this Section, we work with the partition function with free boundary condition.

\subsubsection{The coarse-graining argument}

We cut the system into segments
$(\mathbf{T}_{k-1},\mathbf{T}_{k}]$,
we estimate the contribution of the different segments
separately, and then identify the ones 
that could actually contribute to the free energy.
We give the following  coarse-graining lemma,
that enables us to do so.

\begin{lemma}
\label{lem:coarse}
For any increasing sequence of integers $0=\mathbf{T}_0<\mathbf{T}_1<\mathbf{T}_2<\cdots$, and
every $N\in \N$, one has
\begin{equation}
\label{coarse1}
 Z_{\mathbf{T}_{N},h}^{\go,\gb}\le \prod_{k=1}^N \max_{y\in[0,\mathbf{T}_N]}
 \left[ \sum_{t=1}^{\mathbf{T}_k-\mathbf{T}_{k-1}} \frac{\K(y+t)}{\bar\K(y)} 
   Z_{[\mathbf{T}_{k-1}+t,\mathbf{T}_{k}],h}^{\go,\gb} 
   + \frac{\bar \K(y+\mathbf{T}_k-\mathbf{T}_{k-1})}{\bar \K(y)}\right].
\end{equation}
One also gets the rougher bound
\begin{equation}
\label{coarse2}
 Z_{\mathbf{T}_{N},h}^{\go,\gb}\le \prod_{k=1}^N \left[\left(\max_{x\in (\mathbf{T}_{k-1}, \mathbf{T}_{k}]}
Z^{\go,\gb}_{[x,T_{i}],h}\right)\vee 1\right].
\end{equation}
\end{lemma}


We repeat here that a natural choice, that we make,
is to define the sequence $(\mathbf{T}_k)_{j\geq0}$
iteratively: $\mathbf{T}_1=\mathbf{T}_1^{u,\gb}(\go)$, and then
for any $k\geq1$, one has $\mathbf{T}_{k+1}=\mathbf{T}_1\left(\theta^{\mathbf{T}_k}\go\right)$,
so that the sequence $(\mathbf{T}_{k+1}-\mathbf{T}_{k})_{k\geq0}$ is ergodic.
Then, with Lemma \ref{lem:coarse}, one gets that 
\begin{multline}
\label{upgeneralF}
\tf(\gb,h) = \lim_{N\to\infty} \frac{1}{\mathbf{T}_N} \log  Z_{\mathbf{T}_N,h}^{\go,\gb}\\
   \leq \lim_{N\to\infty} \frac{N}{\mathbf{T}_N} \frac1N
      \sum_{k=1}^{N} \log\left( \max_{y\in\bbN} \left\{ \sum_{t=1}^{\mathbf{T}_k-\mathbf{T}_{k-1}} \frac{\K(y+t)}{\bar\K(y)} 
   Z_{[\mathbf{T}_{k-1}+t,\mathbf{T}_{k}],h}^{\go,\gb} 
   + \frac{\bar \K(y+\mathbf{T}_k-\mathbf{T}_{k-1})}{\bar \K(y)}\right\}\right) \\
        = \frac{1}{\bbE[\mathbf{T}_1]}
     \bbE\left[ \log \left( \max_{y\in\bbN} \left\{ \sum_{t=1}^{\mathbf{T}_1} \frac{\K(y+t)}{\bar\K(y)} 
   Z_{[t,\mathbf{T}_{1}],h}^{\go,\gb} 
   + \frac{\bar \K(y+\mathbf{T}_1)}{\bar \K(y)}\right\} \right)\right]
\end{multline}
where we used twice Birkhoff's Ergodic Theorem for the last equality.

We are therefore reduced to estimate the partition function on the first segment $(0,\mathbf{T}_1]$,
more precisely all the $Z_{[t,\mathbf{T}_{1}],h}^{\go,\gb}$.
Note that one also gets a rougher bound, using the second part of Lemma \ref{lem:coarse}:
\begin{equation}
\label{upsimpleF}
 \tf(\gb,h) \leq \frac{1}{\bbE[\mathbf{T}_1]}
     \bbE\left[ \log \left( \left(\max_{t\in[0,\mathbf{T}_1]}  Z_{[t,\mathbf{T}_{1}],h}^{\go,\gb} \right)\vee 1\right)\right].
\end{equation}

\begin{proof}[Proof of Lemma \ref{lem:coarse}]
The proof is similar to what is done in \cite[Lem. 5.2]{BLpin}, and
we sketch it briefly.
The proof is done by induction. The case $N=1$ is trivial.

Then, one writes
\begin{equation}
\label{proofcoarse1}
  \frac{Z_{\mathbf{T}_{N+1},h}^{\go,\gb}}{Z_{\mathbf{T}_{N},h}^{\go,\gb} }
 = \bE_{\mathbf{T}_N,h}^{\go,\gb} \left[ \exp\left( 
       \sum_{j=\mathbf{T}_N+1}^{\mathbf{T}_{N+1}}  (\gb\go+h) \ind_{\{j\in\tau\}} \right)  \right].
\end{equation}
One decomposes according to the positions of the last renewal before  $\mathbf{T}_N$ (noted $\mathbf{T}_N-y$) and
the first one after $\mathbf{T}_N$ (noted $\mathbf{T}_N+t$) to get, as in \cite[Eq. (4.31)]{BLpin}
\begin{multline}
\label{proofcoarse2}
\bE_{\mathbf{T}_N,h}^{\go,\gb} \left[ \exp\left( 
       \sum_{j=\mathbf{T}_N+1}^{\mathbf{T}_{N+1}}  (\gb\go+h) \ind_{\{j\in\tau\}} \right)  \right]\\
       \leq \max_{y\in[0,\mathbf{T}_N]}  
  \left[  \sum_{t=1}^{\mathbf{T}_{N+1}-\mathbf{T}_N} 
             \frac{\K(y+t)}{\bar\K(y)} Z_{[\mathbf{T}_N+t,\mathbf{T}_{N+1}],h}^{\go,\gb}
         + \frac{\bar\K(y+\mathbf{T}_{N+1}-\mathbf{T}_N)}{\bar\K(y)}  \right].
\end{multline}

The second inequality comes easily from the first one, bounding uniformly
$ Z_{[\mathbf{T}_{k-1}+t,\mathbf{T}_{k}],h}^{\go,\gb}$ by $\max_{x\in (\mathbf{T}_{k-1}, \mathbf{T}_{k}]}
Z^{\go,\gb}_{[x,T_{i}],h}$.
\end{proof}

\subsubsection{Rough coarse-graining: proof of the bound \eqref{simpleboundF}}

We first use a very rough decomposition of our environment, taking $\mathbf{T}_k=T_{2k}$
(recall definition \eqref{defxi}).
Thanks to the coarse-graining lemma \ref{lem:coarse}, and in view of \eqref{upgeneralF},
one is reduced to 
estimate the partition function on the segment $(0,T_{2}]$.
\begin{lemma}
\label{lemZbloc}
 There exist some $C,c_1>0$ and some constant $\mathbf{c}_1>0$ such that for any $\gb\in(0,1)$ and $u\in(0,c_1\gb)$,
\begin{equation}
\label{estimblocs}
 \begin{split}
   \max_{x\in (0, T_{2}]} Z^{\go,\gb}_{[x,T_{2}],-\gb+u} & \leq 1
    \hspace{1.4cm} \text{if } \xi_{1}\leq  \mathbf{c}_1\gb u^{-1},\\
   \max_{x\in (0, T_{2}]} Z^{\go,\gb}_{[x,T_{2}],-\gb+u}& \leq e^{ \xi_{1} u} 
    \hspace{1.9cm} \text{if } \xi_{1}\geq \mathbf{c}_1\gb u^{-1}.
 \end{split}
\end{equation} 
\end{lemma}

Thanks to translation invariance, this tells that
the contribution of a block $(T_{2k}, T_{2k+2}]$ in the decompositions \eqref{coarse1}-\eqref{coarse2} is 
null if $\xi_{2k+1}= T_{2k+1}-T_{2k}\leq \mathbf{c}_1 \gb u^{-1}$, and (possibly) non-zero otherwise.
The bounds \eqref{estimblocs} are equivalent to Lemmas 4.3 (for $\ga>1$) and 5.3
(for $\ga<1$)
in \cite{BLpin}, but in the present case, we deal with the cases
$\ga>1$ and $\ga<1$ at the same time (with however some loss in the case $\ga<1$).

\begin{proof}
On the block $(0, T_{2}]$,
one has $\go\equiv+1$ for $i\in(0,T_{1}]$, and $\go_i=-1$ for $i\in(T_{1},T_{2}]$.
The second inequality is then trivial, using only that $\gd_n\leq1$
and that $\gb\go_i -\gb+u \leq 0$ for $i\in(T_{1},T_{2}]$, provided $u\leq 2\gb$: 
$$\sum_{n=x}^{T_{2}} (\gb\go_n-\gb+u)\gd_n\leq \sum_{n=1}^{T_{1}} u =u\xi_{1} .$$

We therefore focus on the first inequality, in the case $\ga\in(0,1)$
(the case $\ga>1$ being treated in \cite[Lem. 4.2]{BLpin},
or equivalently with the following technique).
If $x\in(T_{1}, T_{2})$, $\go_x=-1$, and then one has for $u\leq \gb$
\begin{equation}
\label{supZbloc1}
Z^{\go,\gb}_{[x,T_{2}],-\gb+u} = e^{-2\gb+u}
   \E\left[\exp\left( (-2\gb+u) \sum_{n=1}^{T_{2}-x} \gd_n \right) \right] \leq e^{-\gb}.
\end{equation}

If $x\in(0, T_{1}]$, since $\go_{T_{1}+1}=-1$, one gets for $u\leq \gb$
\begin{equation}
\label{supZbloc}
Z^{\go,\gb}_{[x,T_{2}],-\gb+u}
\leq e^u \E\left[\exp\left( u \sum_{n=1}^{T_{1}+1-x} \gd_n \right) \exp(-\gb \ind_{\{T_{1}+1-x \in\tau\}}) \right].
\end{equation}

Then, setting $l=T_{1}+1-x$, we notice that
\begin{equation}
 \E\left[e^{ u \sum_{n=1}^{l} \gd_n } e^{-\gb \ind_{\{l \in\tau\}}} \right]\\
   \leq  \E\left[e^{ u \sum_{n=1}^{l} \gd_n } \right]
   - c\gb \E\left[e^{ u \sum_{n=1}^{l} \gd_n } \ind_{\{l \in\tau\}} \right] =
       Z_{l,u}^{\rm pur} -c\gb Z_{l,u}^{\rm pur, \pin},
\end{equation}
where we used that $\gb \ind_{\{l\in\tau\}}\leq 1$ to expand $e^{-\gb \ind_{\{l \in\tau\}}}$.
We are therefore left with estimating the {\sl pinned}
and {\sl free} partition function.
When $l\leq u^{-1}$, we have easy bounds:
\begin{equation}
 Z_{l,u}^{\rm pur} = 1 +\sum_{k=1}^{+\infty}  \frac{u^k}{k!} \bE\left[\left(\sum_{n=1}^l \gd_n\right)^k\right]
   \leq 1 + u \bE\left[\sum_{n=1}^l \gd_n \right] \sum_{k=1}^{+\infty}  \frac{(ul)^{k-1}}{k!}
   \leq 1 + c' u l^{\ga} ,
\end{equation}
and also
\begin{equation}
 Z_{l,u}^{^{\rm pur},\pin}\geq \bP(l\in\tau) \geq c'' l^{\ga -1},
\end{equation} 
where we used twice that $\bP(n\in\tau) \stackrel{n\to\infty}{\sim} cst.\ n^{\ga-1}$, see \cite{Doney}.
Note that these bounds are actually sharp when $l$ is smaller than the correlation length (which is
$\tf(u)^{-1}$, cf. \cite{G_correl}).

All together, if $\xi_{1}\leq u^{-1}$, one has that 
\begin{equation}
 Z^{\go,\gb}_{[x,T_{2}],-\gb+u} \leq e^u \left(1+ c'(\xi_{1})^{\ga}\left(u-c \gb (\xi_{1})^{-1}\right)\right).
\end{equation} 
If $\xi_{1}\leq u^{-1}\times c \gb/2$, one finally has
\begin{equation}
 \max_{x\in (0, T_{2}]} Z^{\go,\gb}_{[x,T_{2}],-\gb+u} 
  \leq e^u \left(1- cst.\ \gb (\xi_{1})^{\ga-1} \right) \leq  \exp\left(u-  cst.\ \gb (\xi_{1})^{\ga-1}\right),
\end{equation}
which is smaller than $1$ provided that $u$ is small enough,
since $(\xi_{1})^{\ga-1}\!\geq c u^{1-\ga} \gg u$.
\end{proof}

In the end, Lemma \ref{lemZbloc} gives
\begin{equation}
\max_{x\in (0, T_{2}]} Z_{[x,T_2],-\gb+u}^{\go,\gb}
\leq  \ind_{\{\xi_1< \mathbf{c}_1\gb u^{-1}\}} +e^{u\xi_1} \ind_{\{\xi_1\geq \mathbf{c}_1\gb u^{-1}\}}.
\end{equation} 
Then, the inequality \eqref{upsimpleF} with $\mathbf{T}_1=T_2$ gives that
\begin{equation}
\label{finupperbound}
\tf(\gb,-\gb+u)\leq \frac{1}{\bbE[T_2]} u \bbE\left[ \xi_1 \ind_{\{\xi_1\geq \mathbf{c}_1\gb u^{-1}\}}\right].
\end{equation}

\subsubsection{Refined coarse-graining: proof of bound \eqref{upboundF}}
\label{sec:upbound}

We focus on the case $\ga<1$ since the previous bound is sufficient for our purposes
in the case $\ga>1$, but the following technique works also in the case $\ga>1$.
The procedure that we apply here is actually the same as the one in \cite[Sec. 5.1]{BLpin},
but we recall it here (briefly), for the sake of completeness.

In view of of the inequality \eqref{upgeneralF},
we only have to estimate the partition function on the first segment $[0,\mathbf{T}_1]$.
We take $\gb\in(0,1)$, and we choose
\begin{equation}
\label{defTbolt}
\begin{split}
 \mathbf{T}_1 & :=T_{L}, \\
 \text{with } L=L(u): &= \lfloor\mathbf{c} \gb^{1/\ga} \tf(u)^{-1} \rfloor,  
\end{split}
\end{equation}
where the constant $\mathbf{c}$ has yet to be chosen (small), independently on $\gb$.
We now give here the key lemma, that deals with the estimate of the partition function,
and is extracted from \cite[Lem. 5.3]{BLpin}.
\begin{lemma}
\label{lemZbloc2}
There exist constants $c_1$ and $\mathbf{c}$ (entering in the definition of $L(u)$), such that
for all $\gb\in(0,1)$ and $u\in [0,c_1\gb)$, we have the two following cases:

If $\mathbf{T}_1< 4\bbE[\mathbf{T}_1] $, there is a constant $C$ (independent of $\gb\in(0,1)$) such that
\begin{equation}
\label{boundZhard}
\max_{y\geq 0}
 \left[ \sum_{t=1}^{\mathbf{T}_1} \frac{\K(y+t)}{\bar \K(y)} 
   Z_{[t,\mathbf{T}_{1}],-\gb+u}^{\go,\gb} 
      + \frac{\bar \K(y+\mathbf{T}_1)}{\bar\K(y)}\right] \leq 1.
\end{equation}

If
$\mathbf{T}_1\geq 4\bbE[\mathbf{T}_1] $, then
\begin{equation}
\label{boundZeasy}
\max_{y\geq 0}
 \left[ \sum_{t=1}^{\mathbf{T}_1} \frac{\K(y+t)}{\bar \K(y)} 
   Z_{[t,\mathbf{T}_{1}],-\gb+u}^{\go,\gb} 
      + \frac{\bar \K(y+\mathbf{T}_1)}{\bar\K(y)}\right] \leq e^{c \gb^{1-\nu^{\rm pur}}\tf(u) \mathbf{T}_1},
\end{equation}
for some constant $c$ (independent of $\gb\in(0,1)$).
\end{lemma}

Inequality \eqref{upsimpleF} then yields 
\begin{equation}
\tf(\gb,-\gb+u) \leq c\gb^{1-\nu^{\rm pur}} \tf(u) 
   \bbE\left[  \frac{\mathbf{T}_1}{\bbE[\mathbf{T}_1]}  \ind_{\{\mathbf{T}_1\geq 4 \bbE[\mathbf{T}_1]\}} \right].
\end{equation}
which is exactly \eqref{upboundF}, once one recalled that $\mathbf{T}_1=T_{L(u)}$,
so that $\bbE[\mathbf{T_1}]$ is of order  $L(u)$.

\begin{proof}

To prove \eqref{boundZeasy}, one just uses the trivial
bound $ Z_{[t,\mathbf{T}_{1}],-\gb+u}^{\go,\gb}
\leq Z_{\mathbf{T_1},u}^{\rm pur}$.
Then, since $\mathbf{T}_1\geq   L(u) \geq \mathbf{c}\gb^{1/\ga} \tf(u)^{-1}$, one uses Lemma \ref{lembasic} ($n=\mathbf{T}_1$, $h=u$, and $\epsilon = \mathbf{c}\gb^{1/\ga}$  is small if $\mathbf{c}$ is small) to get that
$Z_{\mathbf{T_1},u}^{\rm pur}\leq e^{c \gb^{(\alpha-1)/\ga}\tf(u) \mathbf{T}_1 }$ for all $u\in[0,c_1 \gb)$ with suitable constants $c$ and $c_1$.

We now focus on the proof of \eqref{boundZhard}.
We only give the steps of the proof for the sake of completeness,
 since the proof is identical to the one done in \cite[Sec. 5.1]{BLpin}.
Loosely speaking, it relies on the same idea as the previous Section: if the block
$\mathbf{T}_{1}$ is smaller than $4\bbE[\mathbf{T}_1]$, it means that it contains $-1$'s
that are close one to another, and the entropic cost of avoiding them is too high
compared to the energetic gain one has on this segment. 
 
We recall Lemma 4.5 in \cite{BLpin}, that estimates the probability of avoiding a large region in the environment, whatever its shape is.
\begin{lemma}
\label{lememprunte}
There exists some constant $c>0$ such that for any $M>0$, $\varphi>0$,
if one takes $\cS$ a subset of $[1,M]$ of cardinality at least $\varphi M$, one has
\begin{equation}
\bP(\tau\cap \cS \neq \emptyset) \geq c \varphi^{1+\ga}.
\end{equation} 
\end{lemma}

From this Lemma, one is able to estimate the contribution of the different terms
$Z_{[x,\mathbf{T}_1],-\gb+u}^{\go,\gb}$, in the case $\mathbf{T}_1\leq 4 \bbE[\mathbf{T_1}] \leq  4\bbE[T_2]L(u)$.

If $x>L/2$, we use the trivial bound
$Z_{[x,\mathbf{T}_1],-\gb+u}^{\go,\gb}\leq Z_{4\bbE[\mathbf{T}_1], u}^{\rm pur}$.
Since $n=4\bbE[\mathbf{T}_1] \leq c L(u)$, with $L(u)\leq \mathbf{c} \gb^{1/\ga} \tf(u)^{-1}$,
then Lemma \ref{lembasic} gives that there is a constant $c_1$ such that for all $u\in[0,c_1 \gb)$
one has
$Z_{4\bbE[\mathbf{T}_1], u}^{\rm pur} \leq 1+\mathtt{c}\gb $,
where $\mathtt{c}= c \mathbf{c}^{\ga}$ for some constant $c$ ($\mathbf{c}$ being the constant entering in the definition of $L(u)$).
In the end one gets
\begin{equation}
\label{borne1}
\max_{x\in(L/2,\mathbf{T}_1]} Z_{[x,\mathbf{T}_1],-\gb+u}^{\go,\gb} \leq 1+\mathtt{c} \gb.
\end{equation}

If $x\leq L/2\leq T_{L/2}$, as $\mathbf{T}_1=T_{L}$, there are at least $L/4 $ $-1$'s
in the segment $[x,\mathbf{T}_1]$, so that the proportion of $-1$'s is at least
$\frac{L}{4\mathbf{T}_1} \geq (16\bbE[T_2])^{-1}$ (since $\mathbf{T}_1\leq 4 \bbE[T_2] L$).
Take $\varphi=(16\bbE[T_2])^{-1}$,
and $\cS:= \{n\in [1,\mathbf{T}_1-x]\ /\ \go_{x+n} = -1 \}$,
so that applying Lemma \ref{lememprunte}, one has
\begin{multline}
e^{-h}Z_{[x,\mathbf{T}_1],-\gb+u}^{\go,\gb} \leq
  \bE\left[  e^{\sum_{n=1}^{\mathbf{T}_1-x} u \ind_{n\in\tau} } \ind_{\{\tau\cap \cS =\emptyset\}}  \right]
   + e^{-2\gb} \bE\left[  e^{\sum_{n=1}^{\mathbf{T}_1-x} u \ind_{n\in\tau} } \ind_{\{\tau\cap \cS \neq\emptyset\}}  \right]\\
\leq Z_{\mathbf{T}_1-x, u}^{\rm pur} -(1-e^{-2\gb}) \bP(\tau\cap \cS \neq \emptyset)
  \leq  1+\mathtt{c} \gb- c' \varphi^{1+\ga} (1-e^{-\gb}). 
\end{multline}
If the constant $\mathtt{c}=c \mathbf{c}^{\ga}$ is small enough (taking $\mathbf{c}$ in the definition $L(u)$ small),
one has $\mathtt{c} \gb\leq \frac{c'}{2}\varphi^{1+\ga} (1-e^{-\gb})$ for all $\gb\in(0,1)$.
In the end one obtains, for all $\gb\in(0,1)$
\begin{equation}
\label{borne2}
\max_{x\in[0,L/2]} Z_{[x,\mathbf{T}_1],-\gb+u}^{\go,\gb} \leq 1-\frac{c'}{2}\varphi^{1+\ga} (1-e^{-\gb}) \leq 1-2C\gb,
\end{equation}
where the constant $C $ is independent of how small $\mathbf{c}$ is.

Thanks to \eqref{borne1}-\eqref{borne2}
(which are actually Lemma 4.4 in \cite{BLpin}), we have
\begin{multline}
\sum_{t=1}^{\mathbf{T}_1} \frac{\K(y+t)}{\bar \K(y)} 
   Z_{[t,\mathbf{T}_{1}],-\gb+u}^{\go,\gb}
      \leq 
     ( 1-2C\gb) \sum_{t=1}^{L/2} \frac{\K(y+t)}{\bar \K(y)}
        + (1+\mathtt{c}\gb)\sum_{t=L/2}^{\mathbf{T}_1} \frac{\K(y+t)}{\bar \K(y)}\\
        \leq \sum_{t=1}^{\mathbf{T}_1} \frac{\K(y+t)}{\bar \K(y)}  -2C\gb \sum_{t=1}^{L/2} \frac{\K(y+t)}{\bar \K(y)} + \mathtt{c}\gb\sum_{t=1}^{\mathbf{T}_1} \frac{\K(y+t)}{\bar \K(y)}.
\end{multline}
Then, using that $\mathbf{T}_1\leq 4\bbE[T_2] L$
$$\frac{\sum_{t=1}^{L/2} \K(y+t)}{ \sum_{t=1}^{\mathbf{T}_1} \K(y+t)} \geq \frac{\sum_{t=1}^{L/2} \K(y+t)}{ \sum_{t=L/2}^{4\bbE[T_2] L} \K(y+t)} $$
is bounded away from $0$, uniformly in $y$ and $L$. Therefore, if $\mathtt{c}$ is small enough (recall that the value of $C$ is independent from that of $\mathtt{c}$), we have
\begin{equation}
\sum_{t=1}^{\mathbf{T}_1} \frac{\K(y+t)}{\bar \K(y)} 
   Z_{[t,\mathbf{T}_{1}],-\gb+u}^{\go,\gb} \leq \sum_{t=1}^{\mathbf{T}_1} \frac{\K(y+t)}{\bar \K(y)},
\end{equation} 
which gives \eqref{boundZhard}.
\end{proof}

\subsection{On the way to getting a sharper upper bound}
\label{sec:conj}

In both of the previous cases, the coarse-graining procedure
only singled out some favorable stretches, but did not keep track of the cost
from avoiding the repulsing ones.
Therefore, the bounds \eqref{upboundF}-\eqref{upsimpleF}
on the free energy are really rough, and yet sufficient to our purpose.
We present now a procedure to get better bounds on the free energy,
but since the bounds obtained are
hard to be estimated (but may be computed in some particular cases),
we only give a sketch of the method.
The idea is  to use a refined multiscale coarse graining,
dividing the system into larger and larger meta-blocks.
We give a sketch of the proof for the first step of improvement, which follows the idea
of Section 3.2 and 4.2 in \cite{BLpin},
and we explain how one could iterate the procedure.
We then give a heuristic justification of Conjecture \ref{conj}.

\smallskip
{\bf First step of improvement}
We
enlarge the scale of the coarse-graining procedure of Section \ref{sec:upbound}:
define $\mathbf{T}_k^{(1)}:=\cT_k(A_1)$
where $A_1=A_1(u):=\mathbf{c}_1\gb u^{-1}$ is the threshold appearing in Lemma \ref{lemZbloc}.
One then divides the environment into meta-blocks $(\mathbf{T}_k^{(1)}, \mathbf{T}_{k+1}^{(1)}]$.
For each meta-block, there is a certain (random)
number of segments $(T_i, T_{i+1}]$ smaller than $A_1$,
that, in view of Lemma \ref{lemZbloc}, are repulsive (in the sense that the gain on
such a segment is smaller than $1$),
and then there is one "favorable" ending $A_1$-{\sl block}.

Then, one only has to estimate
the partition function on the first meta-block, more precisely all the 
$Z_{[x,\mathbf{T}_1^{(1)}],-\gb+u}^{\go,\gb}$ (remember \eqref{upgeneralF}).
We now call upon Lemma 4.6 in \cite{BLpin}:
it uses
a coarse-graining argument to
keep track that trajectories "avoid" repulsive regions, and jumps directly to the last $A_1$-{\sl block},
the attractive one.
In particular, denoting $k_1=k_1(u)$ the index of the first $A_1$-{\sl block},
so that $\mathbf{T}_1^{(1)}=T_{k_1}$, one obtains
that there exists a constant $\mathrm{C}_1>0$ such that
\begin{equation}
\label{ineqstep1}
Z_{\mathbf{T}_1^{(1)},-\gb+u}^{\go,\gb} \leq e^{u\xi_{k_1}} 
   T_{k_1}^{-\mathrm{C}_1} .
\end{equation}

Then, one has $Z_{\mathbf{T}_1^{(1)},-\gb+u}^{\go,\gb} \! \! >1$ only if
$u\xi_{k_1} \geq \mathrm{c}_1 \log ( 1+T_{k_1-1} )$,
or more easily $u\xi_{k_1} \geq \mathrm{c}_1 \log k_1 $.
One would therefore have
\begin{equation}
\label{upFstep1}
\tf(\gb,-\gb+u)
         \leq c u\bbP\big(\xi_1\geq A_1(u)\big)
            \bbE\left[ \xi_{k_1}\ind_{\{\xi_{k_1} \geq \mathrm{c}_1 u^{-1} \log  k_1 \}}  \right]
\end{equation}
where we used that $\bbE[\mathbf{T}_1^{(1)}]=\bbE[\cT_1(A_1)]= c' \bbP(\xi_1\geq A_1(u))^{-1} $.

\begin{rem}\rm
This presumably gives a sharper upper bound
to the free energy than \eqref{upboundF} (or \eqref{roughF})
since the condition $\left\{\xi_{k_1} \geq \mathrm{c}_1 u^{-1} \log k_1\right\}$
is restrictive. The difficulty is actually to
estimate the influence of this condition in the case of a very general environment.
If the sizes of the blocks were independent (which is the case in \cite{BLpin}),
then the integer $k_1$ is a geometric random variable
of parameter $\bbP(\xi_1\geq A_1)$,
and one is able to estimate precisely the upper bound \eqref{upFstep1}.
This gives for example \cite[Th. 2.1]{BLpin}.
\end{rem}

\begin{rem}\rm
The method sketched here is valid for any $\ga>0$: it relies
on Lemma \ref{lemZbloc} to identify repulsive and attractive blocks $(T_i,T_{i+1}]$.
However, in the case $\ga<1$, one can  identify
repulsive and attractive segments $(\mathbf{T}_i,\mathbf{T}_{i+1}]$,
thanks to Lemma \ref{lemZbloc2} (with $\mathbf{T}_i$ defined in \eqref{defTbolt}),
and then obtain a refined bound.
\end{rem}

\smallskip
{\bf Iterating the procedure.}
The above reasoning tells, among
the segments $(\mathbf{T}_{k}^{(1)},\mathbf{T}_{k+1}^{(1)}]$,
which are the favorable ones
($\xi_{k_1}\geq \mathrm{c}_1 u^{-1}\log k_1$, as far as the first block is concerned)
and which are the unfavorable ones ($\xi_{k_1}< \mathrm{c}_1 u^{-1} \log k_1$).
One is then able to repeat the procedure presented above at larger scales:

 \textbullet\  construct $\big(\mathbf{T}_k^{(2)},\mathbf{T}_{k+1}^{(2)}\big]$,
  composed of (many) blocks $\big(\mathbf{T}_j^{(1)},\mathbf{T}_{j+1}^{(1)}\big]$
  identified  as \emph{globally repulsive},
  and then one \emph{globally attractive} ending block;
  
 \textbullet\  identify when the blocks $\big(\mathbf{T}_k^{(2)},\mathbf{T}_{k+1}^{(2)}\big]$
 are favorable or unfavorable (justifying
  that the main contribution comes from trajectories that jump over the
  first repulsive blocks $\big(\mathbf{T}_j^{(1)},\mathbf{T}_{j+1}^{(1)}\big]$,
   and aim directly at the last, favorable one,
  to obtain an analogue of \eqref{ineqstep1});
  
  \textbullet\ repeat the argument to enlarge the scale...

\smallskip
In the end, one obtains sharper and sharper bounds on the free energy,
the difficulty being to give simple estimates of the bounds.

\smallskip

{\bf Heuristic justification of the Conjecture \ref{conj}.}
One realizes that, after $j$ steps of improvements,
the main contribution in the first block $(0,\mathbf{T}_1^{(j)}]$ 
comes from trajectories aiming directly at its last  $A$-\textit{block}.
Then, after $j$ steps, if $k_1^{(j)}$ is such that $\mathbf{T}_1^{(j)}=T_{k_1^{(j)}}$, the first block is considered
attractive if $u\xi_{k_1^{(j)}}\geq  \mathrm{c} \log T_{k_1^{(j)}}$.
This consideration meets the definition \eqref{defAgo}:
$A_{\mathrm{c} u}^{\go}$ is the first length $\xi_k$
 such that $u\xi_k \geq \mathrm{c} \log T_k $,
and the segment $(0,\cT_1(A_{\mathrm{c} u}^{\go})]$ is the first one to be favorable (with main contribution coming from trajectories aiming directly at the last $A$-\textit{block}).

If $A_{\mathrm{c} u}^{\go}<+\infty$, there is some step $j<+\infty$
such that the first block $(0,\mathbf{T}_1^{(j)}]$ is attractive: this justifies that the localization strategy sketched in Section \ref{sec:bounds} should be the right one. 
On the other hand, if $A_{\mathrm{c} u_0}^{\go}=+\infty$, then
the first segment $(0,\mathbf{T}_1^{(j)}]$ remains unfavorable after any step $j$ (with $\mathbf{T}_1^{(j)}$
going to infinity as $j$ goes to infinity).
It would confirm that, in absence of infinite disorder
(recall its different characterizations, in particular in terms of $A_{\gd}^{\go}$, see Lemma \ref{lem:charactstrong}), there should be
some $u_0$ sufficiently small such that
$\liminf_{n\to\infty} Z_{n,-\gb+u_0}^{\go,\gb}\leq 1$, so that $\tf(\gb,-\gb+u_0)=0$ and $h_c(\gb)\geq-\gb+u_0$.

\medskip

{\bf Acknowledgment.}
 We deeply thank F.L. Toninelli for his constant support in this project,
 and his numerous and precious advices and proof readings, as well as
 H. Lacoin who was of great help on the manuscript.
 This work was initiated during the author's doctorate at the Physics Department of \'Ecole Normale
 Sup\'erieure de Lyon, and its hospitality and support is gratefully acknowledged.

\begin{appendix}
\section{On the sign of Gaussian sequences}
\label{app}

Let $\W=\{\W_n\}_{n\in\N}$ be a stationary Gaussian process, centered and with unitary variance,
and with covariance matrix denoted by $\gU$, the covariance function being $\rho_k = \gU_{i,i+k}$.
We also denote $\gU_l$ the covariance matrix of the vector $(\W_1,\ldots,\W_l)$, which is just a restriction
of $\gU$.
We recall Assumption \eqref{assumpgauss}:
correlations are non-negative, and power-law decaying,
\begin{equation}
 \rho_{k}\stackrel{k\to\infty}{ \sim} c  k^{- a}, \quad \text{for some } a>0 \text{ and } c>0.
\end{equation}

\begin{proof}{Proof of Proposition \ref{app:gauss}}

We recall here
a more general lower bound, dealing with the probability for a 
Gaussian vector to be componentwise larger than some given value.
\begin{lemma}[\cite{Bcorrel}, Lemma A.3]
Under assumption \eqref{assumpgauss} with $a<1$,
there exist two constants $c,C>0$ such that for every $l\in\N$, one has
\begin{equation}
 \bbP\left( \forall i\in\{1,\ldots,l\},\ \W_i\geq A \right)\geq
      c^{-1}\, \exp\left( -c \big( A \vee C\sqrt{\log l} \big) ^2  l^{ a} \right).
\end{equation}
\label{lem:shiftgauss}
\end{lemma}
Taking $A=0$ in this Lemma gives \eqref{gausslow}.

To simplify notations, we prove the upper bound for the specific sequence $i_k=k$,
$k\in\{1,\ldots,n\}$. The general proof follows the same reasoning.
One first observes that for any subset $\{k_1,\ldots,k_m\}\subset \{1,\ldots,n\}$, $m\in\N$,
one has
\begin{equation}
\label{deleteXi}
 \bbP\left( \W_i\geq 0 \, ;\, \forall i\in \{1,\ldots,n\}\right)\leq
    \bbP\left( \W_{k_j}\geq 0 \, ;\, \forall j\in \{1,\ldots,m\}\right)
\end{equation} 
The idea is that, if the $k_j$'s are sufficiently far one from another,
the Gaussian vector $(\W_{k_1},\ldots,\W_{k_m})$ behaves like
an independent one.

\begin{claim}
$\bullet$ If $ a<1$, then there exists some $A>0$ such that
taking $k_j:=j\lfloor A n^{1- a}\rfloor$ for $j\in\{0,\ldots,m:= \lceil A^{-1} n^{ a} \rceil\}$,
one has some constant $c>0$ such that for all $n\in\bbN$
\begin{equation}
 \bbP\left( \W_{k_j}\geq 0 \, ;\, \forall j\in \{1,\ldots,m\}\right) \leq e^{-c m}.
\end{equation}

$\bullet$ If $ a> 1$, then there exists some integer $A>0$ such that
taking $k_j:=j A$ for $j\in\{0,\ldots,m:= \lceil A^{-1}n \rceil\}$,
one has some constant $c>0$ such that for all $n\in\bbN$
\begin{equation}
 \bbP\left( \W_{k_j}\geq 0 \, ;\, \forall j\in \{1,\ldots,m\}\right) \leq e^{-c m}.
\end{equation}
\end{claim}
This claim, together with \eqref{deleteXi}, gives the conclusion. We now prove the claim.

Under $\bbP$, the vector $(\W_{k_1},\ldots,\W_{k_m})$ is a Gaussian vector with covariance matrix
$\tilde \gU_m$, with $\tilde \gU_{ij} = \gU_{k_i,k_j}=\rho_{|k_j-k_i|}$ for $i,j\in\{1,\ldots,m\}$.
We note $\tilde \bbP$ the law of this $m$-dimensional vector.
Then if $\hat \bbP$ denotes the law of a $m$-dimensional
independent standard Gaussian vector $\cN(0,{\rm Id})$,
a change of measure procedure gives thanks to the Cauchy-Schwarz inequality
\begin{equation}
  \tilde \bbP\left( \W_{j}\geq 0 \, ;\, \forall j\in \{1,\ldots,m\}\right)\leq
    \left( \frac12 \right)^{m/2}
       \hat\bbE \left[ \left( \frac{\dd \tilde\bbP}{\dd \hat \bbP}\right)^2\right]^{1/2}.
\label{cauchyschwarz}
\end{equation} 
One has
$\frac{\dd \tilde\bbP}{\dd \hat \bbP}(X) = (\det \tilde \gU_m)^{-1/2} e^{-\frac12 \langle (\tilde \gU_m^{-1}-I)X,X\rangle}$
from the definitions of $\tilde\bbP$ and $\hat \bbP$, so that
from a Gaussian computation, one gets
\begin{equation}
\label{boundchangemes}
 \hat\bbE \left[ \left( \frac{\dd \tilde\bbP}{\dd \hat \bbP}\right)^2\right]^{1/2}
  = (\det \tilde\gU_m)^{-1/2}(\det (2 (\tilde\gU_m)^{-1}-I))^{-1/4} = \det(I-V^2)^{-1/4} 
\end{equation}
where we defined $V:=\tilde\gU_m-I$.

We now estimate $\det(I-V^2)$.
Note that the maximal eigenvalue $\tilde \gl$ of $\tilde\gU_m $ verifies
\begin{equation}
\tilde \gl \leq \max_{i\in\{1,\ldots,n\}} \sum_{j=1}^{n} \tilde\gU_{ij} \leq1+2\sum_{p=1}^{m} \rho_{k_p}.
\end{equation}
Then we use the definition of $k_p$ and $m$, and the assumption \eqref{assumpgauss} on $(\rho_k)_{k\geq 0}$.
We get:

\textbullet\ if $ a<1$ one has
$\tilde \gl \leq 1+c A^{- a} n^{- a(1- a)} m^{1- a} \leq 1+ cA^{-1}$,

\textbullet\  if $ a>1$ one has
$\tilde \gl \leq 1+c A^{- a}$.

In both cases one chooses $A$ large enough so that
$\tilde \lambda \leq 3/2$.
Thus the eigenvalues of $I-V^{2}$ are bounded from below by $ 1-(\tilde \gl-1)^2\geq 3/4$, so that in the end one
has $\det(I-V^2)\geq (3/4)^m$. Combining \eqref{boundchangemes} and \eqref{cauchyschwarz}
one gets
\begin{equation}
 \bar \bbP\left( \W_{j}\geq 0 \, ;\, \forall j\in \{1,\ldots,m\}\right)\leq
    \left( \frac12 \right)^{m/2} \left( \frac34 \right)^{-m/4}\leq 3^{-m/4}. 
\end{equation}
\end{proof}

\section{Estimate on the homogeneous model}
\label{app2}

The following Lemma tells that $n\approx\tf(h)^{-1}$ is the threshold at which
the partition function starts growing exponentially.
\begin{lemma}
\label{lembasic}
We take $\ga\neq 1$. There exist
$\epsilon>0$, and constants $c, c', c''>0$
not depending on $\epsilon$, such that
if $0<h\leq c'' \epsilon ^{\ga\wedge 1} \leq 1$, then 

\textbullet\ one has
$Z_{n,h}^{\rm pur}\leq 1+c (n\tf(h))^{\ga \wedge 1} \leq 1+c \epsilon^{\ga\wedge 1}$
 for every $n\leq\epsilon \tf(h)^{-1}$;

\textbullet\  one has
$Z_{n,h}^{\rm pur} \leq \exp( c' \epsilon ^{\ga \wedge 1 -1} \tf(h) n)$
 for every $n\geq\epsilon \tf(h)^{-1}$.
\end{lemma}

\begin{proof}
Because the case $\ga>1$ is easier (essentially since $\tf(h)$ is proportional to $h$),
we restrict to the case $\ga<1$.

Using that $\bP(|\tau\cap[0,n]|\geq k) \leq (1-\bar \K(n))^k$,
one writes
\begin{equation}
\label{basic1}
Z_{n,h}^{\rm pur} = 1+ \sum_{k=1}^{n} (e^{kh} - e^{(k-1)h}) \bP(|\tau\cap[0,n]|\geq k)
  \leq 1+ \sum_{k=1}^{n_0} h \left( e^h (1- \bar\K(n))\right)^k.
\end{equation}
Then, there is some constant $cst.>0$ such that $\bar \K(n) \geq cst. n^{-\ga} $. Since $h\in [0,1]$, we have that $e^h (1- \bar \K(n)) \leq 1 +2h- cst. n^{-\alpha}$. Since $\tf(h)$ is of order $h^{1/\ga}$, we get that, uniformly for $n\leq \epsilon \tf(h)^{-1}$, one has $n^{-\ga} \geq \epsilon^{-\ga} h$. Therefore, if $\epsilon$ is small enough, $e^h (1- \bar \K(n)) \leq 1 - \frac12 cst. n^{-\alpha}<1$. 
From \eqref{basic1} one then gets
\begin{equation}
Z_{n_0,h}^{\rm pur} \leq 1+ h \sum_{k=1}^{n} \left( 1- \frac12 cst. n^{-\ga} \right)^k
     \leq 1+ \frac{h}{\frac12 cst. n^{-\ga}},
\end{equation}
which gives $Z_{n,h}^{\rm pur}\leq 1+c (n\tf(h))^{\ga \wedge 1}$, since $\tf(h)$ is of order $h^{1/\ga}$. 

\smallskip
For the second point, one uses
that for any two integers $n_1,n_2$, decomposing over the first return time after $n_1$, one gets
that $e^hZ_{n_1+n_2,h}^{\rm pur} \leq e^hZ_{ n_1,h}^{\rm pur} e^hZ_{ n_2,h}^{\rm pur}$.
From this, one gets that for any $p\in\bbN$
\begin{equation}
e^h Z_{p n_0,h}^{\rm pur}  \leq \left( e^h Z_{n_0,h}^{\rm pur} \right)^p
   \leq e^{ p \, \frac{c'}{2} \epsilon ^{\ga }  } = e^{  \frac{c'}{2} \epsilon ^{\ga-1} p\,  n_0 \tf(h)  }  
\end{equation}
where we used the previous bound on $Z_{n_0,h}^{\rm pur} $, the fact that
$h\leq c'' \epsilon ^{\ga}$, and the definition of $n_0= \epsilon  \tf(h)^{-1}$.
The general bound for $n\geq n_0$ relies on the monotonicity of $Z_{n,h}^{\rm pur}$
in $n$.
\end{proof}

\end{appendix}

\bibliography{bibliothese}

\end{document}